\pdfoutput=1
\documentclass[12pt,reqno,final]{amsart}
\usepackage[nopagebackref,hidelinks]{mathieu}

\allowdisplaybreaks[3]
\binoppenalty=\maxdimen
\relpenalty=\maxdimen
\DeclareMathOperator{\id}{id}
\DeclareMathOperator{\rk}{rk}
\DeclareMathOperator{\pbw}{pbw}
\DeclareMathOperator{\Bott}{Bott}
\DeclareMathOperator{\EXP}{EXP}

\newcommand{\NN}{\mathbb{N}} 
\newcommand{\NO}{\mathbb{N}_0} 
\newcommand{\ZZ}{\mathbb{Z}} 
\newcommand{\RR}{\mathbb{R}} 
\newcommand{\CC}{\mathbb{C}} 
\newcommand{\into}{\hookrightarrow}
\newcommand{\onto}{\twoheadrightarrow}
\newcommand{\xto}[1]{\xrightarrow{#1}}
\newcommand{\transpose}{^{\top}} 
\newcommand{\argument}{-}
\newcommand{\abs}[1]{\left| #1 \right|} 
\newcommand{\liederivative}[1]{\mathcal{L}_{#1}}
\newcommand{\shuffle}[2]{\mathfrak{S}_{#1}^{#2}}
\newcommand{\sections}[1]{\Gamma(#1)}
\newcommand{\enveloping}[1]{\mathcal{U}(#1)}
\newcommand{\groupoid}[1]{\mathscr{#1}}
\newcommand{\lie}[2]{[#1,#2]} 
\newcommand{\duality}[2]{\left\langle#1\middle|#2\right\rangle} 
\newcommand{\interior}[1]{\iota_{#1}}
\newcommand{\KK}{\Bbbk}
\newcommand{\torsion}{T^\nabla}
\newcommand{\xnabla}{X^\nabla}
\newcommand{\xtwo}{X_2}
\newcommand{\xhigher}{X_{\geqslant 3}}
\newcommand{\machin}[1]{\lceil #1\rfloor}
\newcommand{\etendu}[1]{#1_{\natural}}
\newcommand{\perturbed}[1]{\breve{#1}}

\title[Fedosov dg manifolds associated with Lie pairs]{Fedosov dg manifolds associated with Lie pairs}

\thanks{Research partially supported by NSF grants 
DMS-1707545, DMS-1406668 and DMS-1101827, and NSA grant H98230-14-1-0153.}

\author{Mathieu Stiénon}
\address{Department of Mathematics, Pennsylvania State University}
\email{stienon@psu.edu} 

\author{Ping Xu}
\address{Department of Mathematics, Pennsylvania State University}
\email{ping@math.psu.edu}

\dedicatory{Dedicated to the memory of our colleague and friend John Roe}

\begin{document}

\begin{abstract}
Given any pair $(L,A)$ of Lie algebroids, 
we construct a differential graded manifold $(L[1]\oplus L/A,Q)$, 
which we call Fedosov dg manifold.
We prove that the cohomological vector field $Q$ constructed on $L[1]\oplus L/A$ 
by the Fedosov iteration method arises as a byproduct of the Poincaré--Birkhoff--Witt 
map established in~\cite{arXiv:1408.2903}. 
Finally, using the homological perturbation lemma, 
we establish a quasi-isomorphism of Dolgushev--Fedosov type: 
the differential graded algebras of functions on the dg manifolds $(A[1],d_A)$ 
and $(L[1]\oplus L/A,Q)$ are homotopy equivalent.
\end{abstract}

\maketitle

\tableofcontents

\section*{Introduction}

Fedosov resolutions --- we call them Fedosov dg manifolds later in the paper --- 
played a key role in globalizing Kontsevich's formality theorem to smooth manifolds \cite{MR2102846}. 
One can expect deformation quantization of geometric objects other than smooth manifolds 
to require the development of analogues of Fedosov resolutions for these other geometric objects.
The leaf space of a foliation on a smooth manifold is one instance of such other geometric objects. 
In general, the leaf space may not be a smooth manifold
--- it is in a certain sense a noncommutative manifold.
However, it can be considered as a particular example of Lie pair.

By a Lie pair $(L,A)$, we mean an inclusion $A\hookrightarrow L$
of Lie algebroids over a smooth manifold $M$.
Lie pairs arise naturally in a number of areas of mathematics
such as Lie theory, complex geometry, and foliation theory.
For instance, a complex manifold $X$ determines a Lie pair over $\CC$
with $L= T_X\otimes\CC$ and $A=T_X^{0,1}$.
A foliation $\mathcal{F}$ on a smooth manifold $M$ determines a Lie pair over $\RR$:
this time $L$ is the tangent bundle to $M$ 
and $A$ is the integrable distribution $T_\mathcal{F}$ on $M$ tangent to the foliation $\mathcal{F}$. 
A $\mathfrak{g}$-manifold also gives rise to a Lie pair in natural way \cite{MR3650387}.

The purpose of this paper is to construct analogues of Fedosov resolutions for Lie pairs. 
More precisely, for any Lie pair $(L,A)$, we present two equivalent constructions of a dg manifold, 
called Fedosov dg manifold, and we establish a quasi-isomorphism of Dolgushev--Fedosov type.
The first construction relies on the Poincaré--Birkhoff--Witt map introduced in \cite{arXiv:1408.2903}, 
a generalized symmetrization map, while the second construction is based on Fedosov's iteration method. 

Given a Lie pair $(L,A)$, the quotient $L/A$ is naturally an $A$-module \cite{MR3439229}.
When $L$ is the tangent bundle to a manifold $M$ and $A$ is an integrable distribution on $M$,
the infinitesimal $A$-action on $L/A$ reduces to the classical Bott
flat connection \cite{MR0362335}.
In~\cite{arXiv:1408.2903}, together with Laurent-Gengoux, we showed that, for any Lie pair over $\RR$, 
each choice of (1) a splitting of the short exact sequence of vector bundles \[ 0\to A \to L \to L/A \to 0 \]
and (2) an $L$-connection $\nabla$ on $L/A$ extending the Bott $A$-connection 
determines an exponential map \[ \exp: L/A\to\groupoid{L}/\groupoid{A} .\] 
Here $\groupoid{L}$ and $\groupoid{A}$ are local Lie groupoids corresponding to the Lie algebroids $L$ and $A$, respectively. 
Considering the (fiberwise) infinite-order jet of this exponential map, 
we obtained an isomorphism of filtered $R$-coalgebras (with $R=C^\infty (M)$)
\[ \pbw:\sections{S(L/A)}\to\tfrac{\enveloping{L}}{\enveloping{L}\sections{A}} ,\]
which we called Poincaré--Birkhoff--Witt map. 
In particular, if $L$ is a Lie algebra $\mathfrak{g}$ and $A$ is the trivial Lie algebra of dimension $0$, 
there exists a natural choice of connection and the resulting $\pbw$ map 
is precisely the symmetrization map $S(\mathfrak{g})\to\enveloping{\mathfrak{g}}$. 
These PBW maps arising from Lie pairs admit an explicit recursive characterization 
valid for Lie pairs over any field $\KK$ of characteristic zero and not just $\RR$. 
Hence these PBW maps can be considered as algebraic formal exponential maps.

Transferring the canonical infinitesimal action of $L$ on the coalgebra 
$\tfrac{\enveloping{L}}{\enveloping{L}\sections{A}}$ 
--- this is an infinitesimal action by coderivations --- 
through the map $\pbw$, we obtain a \emph{flat} 
$L$-connection $\nabla^\lightning$ on $S(L/A)$:
\[ \nabla^\lightning_l(s)=\pbw^{-1}\big(l\cdot\pbw(s)\big) ,\]
for all $l\in\sections{L}$ and $s\in\sections{S (L/A)}$. 
The covariant Chevalley--Eilenberg differential 
\[ d_L^{\nabla^\lightning}: \sections{\Lambda^\bullet L^\vee \otimes \hat{S} \big((L/A)^\vee\big)}
\to \sections{\Lambda^{\bullet+1} L^\vee \otimes \hat{S} \big((L/A)^\vee\big)} \]
of the induced flat $L$-connection on the dual bundle 
$\hat{S}\big((L/A)^\vee\big)$ is a derivation of degree $(+1)$ of the algebra
$\sections{\Lambda^\bullet L^\vee\otimes\hat{S}\big((L/A)^\vee\big)}$ 
of smooth functions on the graded manifold $L[1]\oplus L/A$.
As a consequence, $(L[1]\oplus L/A, d_L^{\nabla^\lightning})$ is a dg manifold.
We prove that, when $\nabla$ is torsion-free,
the homological vector field $d_L^{\nabla^\lightning}$ 
coincides with a homological vector field $Q$ constructed 
by Fedosov's iteration method. 
We elect to call a dg manifold $(L[1]\oplus L/A,Q)$ 
constructed in this way a \emph{Fedosov dg manifold}.

It is a well known theorem of Dolgushev~\cite{MR2102846} that, 
for a smooth manifold $M$, 
the Fedosov dg manifold $T_M[1]\oplus T_M$ 
(associated with the Lie pair $(L,A)$ 
where $L$ is the tangent bundle to $M$ 
and $A$ is its trivial subbundle of rank $0$) 
gives rise to a resolution 
$\Omega^\bullet\big(M;\hat{S}(T^\vee_M)\big)$ of $C^\infty(M)$.
Our second main theorem extends this result to Lie pairs. 
Note that, for a Lie pair $(L,A)$, 
the space of functions on the Fedosov dg manifold 
$(L[1]\oplus L/A, Q)$ is the differential graded algebra 
$\big(\sections{\Lambda^\bullet L^\vee\otimes\hat{S}(B^\vee)},Q\big)$ 
while the space of functions on the dg manifold $(A[1],d_A)$ is 
the differential graded algebra $\big(\sections{\Lambda^\bullet A^\vee},d_A\big)$.
We construct an explicit quasi-isomorphism of Dolgushev--Fedosov type 
from $\big(\sections{\Lambda^\bullet A^\vee},d_A\big)$ 
to $\big(\sections{\Lambda^\bullet L^\vee\otimes\hat{S}(B^\vee)},Q\big)$.
More precisely, using homological perturbation, we establish a contraction of 
$\big(\sections{\Lambda^\bullet L^\vee\otimes\hat{S}(B^\vee)},Q)$
onto $\big(\sections{\Lambda^\bullet A^\vee},d_A\big)$:
\[ \begin{tikzcd}
\big(\sections{\Lambda^\bullet A^\vee},d_{A}\big)
\arrow[r, "\perturbed{\tau}", shift left] &
\big(\sections{\Lambda^\bullet L^\vee\otimes\hat{S}(B^\vee)},d_{L}^{\nabla^\lightning}\big)
\arrow[l, "\sigma", shift left]
\arrow[loop, "\perturbed{h}", out=5, in=-5, looseness = 3]
\end{tikzcd} .\]

As an application, we obtain an alternative proof of
a theorem of Emmrich--Weinstein \cite[Theorem~1.6]{MR1327535}.
Given a smooth manifold $M$ and a torsion-free affine connection $\nabla$ on it, 
Emmrich--Weinstein~\cite{MR1327535} constructed a dg manifold $(T_M[1]\oplus T_M,Q)$ 
using Fedosov's iteration method (see~\cite{MR2102846,MR1293654}).
Emmrich--Weinstein~\cite{MR1327535} explained that (1) the derivation 
\[ Q:\Omega^\bullet\big(M;\hat{S}(T^\vee_M)\big)\to\Omega^{\bullet+1}\big(M;\hat{S}(T^\vee_M)\big) \] 
determines a formal flat (nonlinear) Ehresmann connection on some neighborhood of
the zero section of $TM \to M$ and (2) the leaves of this flat Ehresmann connection 
are transversal to the zero section. Hence, this formal flat Ehresmann connection 
induces a `formal exponential map' $\EXP$ --- see~\cite[Section~7]{MR1327535}.
Emmrich--Weinstein proved that the map $\EXP$ coincides with the infinite-order jet 
of the geodesic exponential map $\exp$ associated with the affine connection $\nabla$ 
--- see~\cite[Theorem 1.6]{MR1327535}.
Their proof resorted to an indirect argument involving analytic manifolds.
In this paper, we present a simple and direct proof based on
(1) our result that the homological vector fields $Q$ and $d^{\nabla^\lightning}$ are equal, 
(2) the contraction 
\[ \begin{tikzcd}
C^\infty(M) \arrow[r, "\perturbed{\tau}", shift left] &
\big(\Omega^\bullet(M;\hat{S}(T^\vee_M)),d^{\nabla^\lightning}\big) 
\arrow[l, "\sigma", shift left]
\arrow[loop, "\perturbed{h}", out=5, in=-5, looseness = 3]
\end{tikzcd} \]
mentioned earlier, 
and (3) the geometric interpretation of the PBW map described at length in~\cite{arXiv:1408.2903}.
Indeed, when $L=T_M$ and $A$ is its trivial subbundle of rank $0$, the map 
\[ \perturbed{\tau}:C^\infty(M)\to\Omega^0\big(M;\hat{S}(T_M^{\vee})\big) \]
is precisely the pull-back by the formal exponential map $\EXP$ 
studied by Emmrich--Weinstein~\cite{MR1327535}. 

In fact, we obtain an extension of the Emmrich--Weinstein theorem to the context of matched pairs 
--- see Theorem~\ref{Montmartre}.
A matched pair of Lie algebroids $L=A\bowtie B$
is a Lie pair $(L,A)$ admitting a splitting $j:B\to L$ 
of the short exact sequence $0\to A\to L\to B\to 0$, 
whose image $j(B)$ happens to be a Lie subalgebroid of $L$.
In the special case of matched pairs, 
we obtain an explicit formula for the map $\perturbed{\tau}$ 
--- see Equation~\eqref{eq:FCO} --- 
generalizing Emmrich--Weinsteins's interpretation of $\perturbed{\tau}$ 
(the pull-back by $\text{EXP}$ in the terminology of~\cite{MR1327535})
as the infinite-order jet of an exponential map.

The Dolgushev--Fedosov type resolutions for Lie pairs, which we establish in the present work, 
play a crucial role in the proof of two results expounded in a subsequent work \cite{MR3964152}: 
a formality theorem and a Kontsevich--Duflo type theorem for Lie pairs.
While the spaces of polyvector fields and polydifferential operators on a smooth manifold 
both carry obvious dgla structures, 
there is generally no such obvious $L_\infty$ algebra structure on either of the spaces 
of polyvector fields and polydifferential operators associated with a Lie pair. 
However, there exist natural $L_\infty$ algebra structures on the spaces of polyvector fields and 
polydifferential operators on a dg foliation of the Fedosov dg manifold arising from the Lie pair. 
Our Dolgushev--Fedosov resolutions for Lie pairs allow for the homotopy transfer of these 
$L_\infty$ structures from the Fedosov dg manifold to the Lie pair. 
This was done in~\cite{arXiv:1901.04602}, where the dg foliation of the Fedosov dg manifold 
is called Fedosov dg Lie algebroid. 
The Fedosov dg manifold construction was recently extended to $\ZZ$-graded manifolds 
by Liao--Stiénon~\cite{MR3910470} (see also \cite{MR3754617}). 

\section*{Terminology and notations} 

\subsubsection*{Natural numbers} We use the symbol $\NN$ to denote the set of positive integers and the symbol $\NN_0$ for the set of nonnegative integers. 

\subsubsection*{Field $\KK$ and ring $R$} We use the symbol $\KK$ to denote the field of either real or complex numbers. 
The symbol $R$ always denotes the algebra of smooth functions on $M$ with values in $\KK$. 

\subsubsection*{Completed symmetric algebra}
Given a module $\mathscr{M}$ over a ring, the symbol $\hat{S}(\mathscr{M})$ denotes 
the $\mathfrak{m}$-adic completion of the symmetric algebra $S(\mathscr{M})$, 
where $\mathfrak{m}$ is the ideal of $S(\mathscr{M})$ generated by $\mathscr{M}$. 

\subsubsection*{Duality pairing}
For every vector bundle $E\to M$, we define a duality pairing 
\[ \sections{\hat{S}(E^\vee)}\times\sections{S(E)} \to R \] 
by 
\[ \duality{\nu_1\otimes\cdots\otimes\nu_p}{v_1\otimes\cdots\otimes v_q} 
=\begin{cases}\sum_{\sigma\in S_p}\prod_{k=1}^p\duality{\nu_k}{v_{\sigma(k)}} 
& \text{if } p=q , \\ 0 & \text{otherwise.} \end{cases} \] 

\subsubsection*{Multi-indices}
Let $E\to M$ be a smooth vector bundle of finite rank $r$,
let $(\partial_i)_{i\in\{1,\dots,r\}}$ be a local frame of $E$ 
and let $(\chi_j)_{j\in\{1,\dots,r\}}$ be the dual local frame of $E^\vee$.
Thus, we have $\duality{\chi_i}{\partial_j}=\delta_{i,j}$.
Given a multi-index $I=(I_1,I_2,\cdots,I_r)\in\NO^r$, 
we adopt the following multi-index notations:
\begin{gather*}
I! = I_1! \cdot I_2! \cdots I_r! \\ 
\abs{I} = I_1+I_2+\cdots+I_r \\ 
\partial^I=\underset{I_1 \text{ factors}}{\underbrace{\partial_1\odot\cdots\odot \partial_1}}
\odot\underset{I_2 \text{ factors}}{\underbrace{\partial_2\odot\cdots\odot\partial_2}}
\odot\cdots\odot\underset{I_r \text{ factors}}{\underbrace{\partial_r\odot\cdots\odot\partial_r}} \\ 
\chi^I=\underset{I_1 \text{ factors}}{\underbrace{\chi_1\odot\cdots\odot\chi_1}}
\odot\underset{I_2 \text{ factors}}{\underbrace{\chi_2\odot\cdots\odot\chi_2}}
\odot\cdots\odot\underset{I_r \text{ factors}}{\underbrace{\chi_r\odot\cdots\odot\chi_r}}
\end{gather*}
We use the symbol $e_k$ to denote the multi-index all of whose 
components are equal to $0$ except for the $k$-th which is equal to $1$. 
Thus $\chi^{e_k}=\chi_k$.

\subsubsection*{Shuffles} 
A $(p,q)$-shuffle is a permutation $\sigma$ of the set $\{1,2,\cdots,p+q\}$ such that 
\[ \sigma(1)<\sigma(2)<\cdots<\sigma(p) 
\qquad\text{and}\qquad 
\sigma(p+1)<\sigma(p+2)<\cdots<\sigma(p+q) .\]
The symbol $\shuffle{p}{q}$ denotes the set of $(p,q)$-shuffles. 

\subsubsection*{Graduation shift} 
Given a graded vector space $V=\bigoplus_{k\in\ZZ}V^{(k)}$, the notation $V[i]$ denotes the graded
vector space obtained by shifting the grading on $V$ according to the rule
$(V[i])^{(k)}=V^{(i+k)}$. Accordingly, if $E=\bigoplus_{k\in\ZZ}E^{(k)}$
is a graded vector bundle over $M$, the notation $E[i]$ denotes
the graded vector bundle obtained by shifting the degree in the fibers
of $E$ according to the above rule. 

\subsubsection*{Dg manifolds}
A dg manifold is a $\ZZ$-graded manifold endowed with a homological vector field, 
i.e.\ a vector field $Q$ of degree $(+1)$ such that $\lie{Q}{Q} =0$. 
Dg manifolds are also known as $Q$-manifolds. 
For details and further references, see~\cite{MR1432574, MR2709144, MR2534186}.

\section{Preliminaries}

\subsection{Lie algebroids and Lie pairs}

\subsubsection*{Lie algebroids}
We use the symbol $\KK$ to denote either of the fields $\RR$ and $\CC$. 
A \emph{Lie algebroid} over $\KK$ is 
a $\KK$-vector bundle $L\to M$ together with 
a bundle map $\rho:L\to T_M\otimes_\RR \KK$ 
called \emph{anchor} and a Lie bracket 
$[\argument,\argument]$ on sections of $L$ such that
$\rho:\sections{L}\to\mathfrak{X}(M)\otimes \KK$ is a morphism of Lie algebras
and \[ [X,fY]=f[X,Y]+\big(\rho(X)f\big)Y \]
for all $X,Y\in\Gamma(L)$ and $f\in C^\infty(M,\KK)$.
In this paper `Lie algebroid' always means `Lie algebroid over $\KK$' unless specified otherwise.
A $\KK$-vector bundle $L\to M$ is a Lie algebroid if and only if $\sections{L}$ 
is a \emph{Lie--Rinehart algebra} \cite{MR0154906} over the commutative ring $C^\infty(M,\KK)$. 

\subsubsection*{Lie pairs}
By a \emph{Lie pair} $(L,A)$, we mean an inclusion $A\hookrightarrow L$ 
of Lie algebroids over a smooth manifold $M$.

\begin{examples}\strut
\begin{enumerate}
\item Let $\mathfrak{h}$ be a Lie subalgebra of a Lie algebra $\mathfrak{g}$. 
Then $(\mathfrak{g},\mathfrak{h})$ is a Lie pair over the one-point manifold $\{*\}$. 
\item Let $X$ be a complex manifold. Then $(T_X\otimes\CC,T^{0,1}_X)$ is a Lie pair over $X$. 
\item Let $\mathcal{F}$ be a foliation on a smooth manifold $M$. 
Then $(T_M,T_{\mathcal{F}})$ is a Lie pair over $M$. 
\end{enumerate}
\end{examples}

\subsubsection*{Matched pairs}
A \emph{matched pair of Lie algebroids} $L=A\bowtie B$ is a Lie pair $(L,A)$ together with a splitting 
$j:B\to L$ of the short exact sequence $0\to A\to L\to B\to 0$, 
whose image $j(B)$ happens to be a Lie subalgebroid of $L$ --- 
see~\cite{MR1430434, MR1460632, MR1716681} for more details. 

\begin{examples}\strut
\begin{enumerate}
\item If $X$ is a complex manifold, 
then $T_X\otimes\CC=T^{0,1}_X\bowtie T^{1,0}_X$ 
is a matched pair of complex Lie algebroids over $X$. 
\item
Let $G$ be a Poisson Lie group and let $P$ be a Poisson $G$-space, 
i.e.\ a Poisson manifold $(P,\pi)$ endowed with a $G$-action $G\times P\to P$ 
which happens to be a Poisson map. According to Lu~\cite{MR1430434}, 
the cotangent Lie algebroid $A=\big(T_P^\vee\big)_\pi$ 
and the transformation Lie algebroid $B=P\rtimes\mathfrak{g}$ 
form a matched pair of Lie algebroids over the manifold $P$. 
\end{enumerate}
\end{examples}

\subsection{Chevalley--Eilenberg differentials, connections, and representations}

Let $L$ be a Lie algebroid over a smooth manifold $M$, 
and $R$ be the algebra of smooth functions on $M$ valued in $\KK$.
The Chevalley--Eilenberg differential 
\[ d_L:\sections{\Lambda^k L^\vee}\to\sections{\Lambda^{k+1} L^\vee} \] 
defined by 
\begin{multline*}
\big(d_L \omega\big)(v_0,v_1,\cdots,v_k)=
\sum_{i=0}^{n} \rho(v_i)\big(\omega(v_0,\cdots,\widehat{v_i},\cdots,v_k)\big) \\ 
+\sum_{i<j}\omega(\lie{v_i}{v_j},v_0,\cdots,\widehat{v_i},\cdots,\widehat{v_j},\cdots,v_k) 
\end{multline*}
and the exterior product make $\bigoplus_{k\geqslant 0}\sections{\Lambda^k L^\vee}$ 
into a differential graded commutative algebra.

The following proposition is an immediate consequence of the definitions. 
\begin{proposition}\label{Piacenza}
Let $L$ be a Lie algebroid and let $A$ and $B$ be two vector subbundles of $L$ such that $L=A\oplus B$. 
Let $p:L\onto A$ and $q:L\onto B$ denote the canonical projections, 
let $p\transpose:A^\vee\into L^\vee$ and $q\transpose:B^\vee\into L^\vee$ denote their respective dual maps, 
and set \[ \Omega^{u,v}=\sections{p\transpose(\Lambda^u A^\vee)\wedge q\transpose(\Lambda^v B^\vee)} .\]
\begin{enumerate}
\item If neither $A$ nor $B$ is a Lie subalgebroid of $L$, then 
\[ d_L(\Omega^{u,v})\subset \Omega^{u+2,v-1}\oplus\Omega^{u+1,v}\oplus
\Omega^{u,v+1}\oplus\Omega^{u-1,v+2} .\] 
\item If $A$ is a Lie subalgebroid of $L$, i.e.\ if $(L,A)$ is a Lie pair, then 
\[ d_L(\Omega^{u,v})\subset \Omega^{u+1,v}\oplus
\Omega^{u,v+1}\oplus\Omega^{u-1,v+2} .\] 
\item If both $A$ and $B$ are Lie subalgebroids of $L$, i.e.\ if $L=A\bowtie B$ is a matched pair, 
then \[ d_L(\Omega^{u,v})\subset \Omega^{u+1,v}\oplus\Omega^{u,v+1} .\] 
\end{enumerate}
\end{proposition}

Now let $E\xto{\varpi}M$ be a vector bundle over $\KK$.
The traditional description of a (linear) $L$-connection on $E$ is in terms of a \emph{covariant derivative} 
\[ \sections{L}\times\sections{E}\to\sections{E}: (l,e)\mapsto \nabla_l e \] 
characterized by the following two properties: 
\begin{gather}
\nabla_{f\cdot l} e=f\cdot \nabla_l e , \label{Molenbeek} \\ 
\nabla_l (f\cdot e)=\rho(l)f\cdot e+f\cdot\nabla_l e \label{Anderlecht}
,\end{gather}
for all $l\in\sections{L}$, $e\in\sections{E}$, and $f\in R$.

\begin{remark}
An $L$-connection $\nabla$ on $E$ induces a covariant derivative 
\[ \nabla:\sections{L}\times\sections{S(E)}\to\sections{S(E)} \]
through the Leibniz rule
\[ \nabla_l (b_1\odot\cdots\odot b_n) = \sum_{k=1}^n 
b_1\odot\cdots\odot b_{k-1}\odot\nabla_l b_k\odot b_{k+1}\odot\cdots\odot b_n ,\] 
for all $l\in\sections{L}$ and $b_1,\dots,b_n\in\sections{E}$.
\end{remark}

\begin{remark}
\label{Amalfi}
A covariant derivative 
\[ \nabla:\sections{L}\times\sections{S(E)}\to\sections{S(E)} \] 
induces a covariant derivative 
\[ \nabla:\sections{L}\times\sections{\hat{S}(E^\vee)}\to\sections{\hat{S}(E^\vee)} \] 
through the relation
\[ \rho(l)\duality{\sigma}{s}=\duality{\nabla_l\sigma}{s}+\duality{\sigma}{\nabla_l s} \] 
for all $l\in\sections{L}$, $s\in\sections{SE}$, 
and $\sigma\in\sections{\hat{S}(E^\vee)}$. 
\end{remark}

A \emph{representation of a Lie algebroid} $L$ on a vector bundle $E\to M$ is 
a flat $L$-connection $\nabla$ on $E$, i.e.\ a covariant derivative 
$\nabla:\sections{L}\times\sections{E}\to\sections{E}$ satisfying 
\begin{equation}\label{Jette}
\nabla_{a_1}\nabla_{a_2} e-\nabla_{a_2}\nabla_{a_1} e=\nabla_{\lie{a_1}{a_2}}e ,
\end{equation} 
for all $a_1,a_2\in\sections{L}$ and $e\in\sections{E}$. 
A vector bundle endowed with a representation of the Lie algebroid $L$ is called an \emph{$L$-module}. 
More generally, given a left $R$-module $\mathscr{M}$, by an {\em infinitesimal action} of $L$ on
$\mathscr{M}$, we mean a $\KK$-bilinear map $\nabla:\sections{L}\times\mathscr{M}\to\mathscr{M}$, $(a,e)\mapsto\nabla_a e$ 
satisfying Equations~\eqref{Molenbeek}, \eqref{Anderlecht}, and~\eqref{Jette}. 
In other words, $\nabla$ is a representation of the Lie--Rinehart algebra $(\sections{L},R)$ \cite{MR0154906}.

\begin{example}[\cite{MR3439229}]
\label{Auderghem}
Let $(L,A)$ be a Lie pair.
The \emph{Bott representation} of $A$ on the quotient $L/A$ is the flat connection defined by 
\[ \nabla^{\Bott}_a q(l)=q\big(\lie{a}{l}\big), \quad\forall a\in\sections{A},l\in\sections{L} ,\] 
where $q$ denotes the canonical projection $L\onto L/A$. 
Thus the quotient $L/A$ of a Lie pair $(L,A)$ is an $A$-module. 
\end{example}

The Chevalley--Eilenberg covariant differential
associated to a representation $\nabla$ of a Lie algebroid $L\to M$
of rank $n$ on a vector bundle $E\to M$ is the operator
\[ d_L^{\nabla}: \sections{\Lambda^k L^\vee\otimes E}\to\sections{\Lambda^{k+1} L^\vee\otimes E} \]
that takes a section $\omega\otimes e$ of
$\Lambda^k L^\vee\otimes B$ to
\[ d_L^{\nabla}(\omega\otimes e)=(d_L\omega)\otimes e
+\sum_{j=1}^{n}(\nu_j\wedge\omega)\otimes \nabla_{v_j}e ,\]
where $v_1,v_2,\dots,v_n$ and $\nu_1,\nu_2,\dots,\nu_n$
are any pair of dual local frames for the vector bundles $L$ and $L^\vee$.
Because the connection $\nabla$ is flat, $d_L^{\nabla}$ is a coboundary 
operator: $d_L^{\nabla}\circ d_L^{\nabla}=0$.
	
\subsection{Torsion-free connections}
\label{eq:dog}

Let $(L,A)$ be a pair of Lie algebroids over $\KK$. 
Consider the short exact sequence of vector bundles 
\begin{equation}\label{Sydney} 
\begin{tikzcd}
0 \arrow{r} & A \arrow{r}{i} & L \arrow{r}{q} & L/A \arrow{r} & 0
\end{tikzcd}
.\end{equation}

An $L$-connection $\nabla$ on $L/A$ is said to extend 
the Bott $A$-representation on $L/A$ (see Example~\ref{Auderghem}) if 
\[ \nabla_{i(a)} q(l)=\nabla^{\Bott}_a q(l)
=q\big(\lie{i(a)}{l}\big) ,\quad\forall a\in\sections{A},l\in\sections{L} .\]

Given an $L$-connection $\nabla$ on $L/A$,
its torsion is the bundle map $\torsion:\Lambda^2 L\to L/A$ defined by 
\[ \torsion(l_1,l_2)=\nabla_{l_1}q(l_2)-\nabla_{l_2}q(l_1)-q\big(\lie{l_1}{l_2}\big), 
\quad\forall l_1,l_2\in\sections{L} .\] 
If $\nabla$ is an $L$-connection on $L/A$
extending the Bott $A$-representation on $L/A$,
its torsion descends to a bundle map 
\[ \beta^\nabla:\Lambda^2(L/A)\to L/A ,\]
making the diagram 
\[ \begin{tikzcd}
\Lambda^2 L \arrow{d}[swap]{q} \arrow{r}{\torsion} & L/A \\ \Lambda^2 (L/A) \arrow{ru} [swap]{\beta^\nabla} &
\end{tikzcd} \]
commute. 
According to~\cite[Lemma~5.2]{arXiv:1408.2903}, 
if $\nabla$ is torsion-free, 
it must be an extension of the Bott $A$-representation on $L/A$.
Torsion-free $L$-connections on $L/A$ always exist 
--- see~\cite[Proposition~5.3]{arXiv:1408.2903}.

\subsection{Universal enveloping algebra of a Lie algebroid}

Let $L$ be a Lie $\KK$-algebroid over a smooth manifold $M$ 
and let $R$ denote the algebra of smooth functions on $M$ taking values in $\KK$. 
By $\enveloping{L}$ we denote the universal enveloping algebra 
of the Lie algebroid $L$ --- see~\cite{MR0154906}.
There is a natural ascending filtration on $\enveloping{L}$:
\begin{equation}\label{Watermael}
\cdots \into \mathcal{U}^{\leqslant n-1}(L) \into \mathcal{U}^{\leqslant n}(L) 
\into \mathcal{U}^{\leqslant n+1}(L) \into \cdots 
\end{equation} 

When the base $M$ of the Lie algebroid $L$ is the one-point space 
so that the only fiber of the vector bundle $L$ is a Lie algebra $\mathfrak{h}$, 
the universal enveloping algebra of the Lie algebroid 
is the universal enveloping algebra of the Lie algebra $\mathfrak{h}$. 
When the Lie algebroid $L$ is the tangent bundle $T_M\to M$, 
its universal enveloping algebra $\enveloping{L}$ is the algebra of differential operators on $M$. 
In general, when $L$ is a Lie algebroid over $\RR$, its universal enveloping algebra $\enveloping{L}$ 
can canonically be identified with the algebra of source-fiberwise differential operators on $\groupoid{L}$ 
invariant under left translations \cite{MR1747916} 
--- here $\groupoid{L}$ is any local Lie groupoid with Lie algebroid $L$.
Similarly, $\bigotimes^{\bullet}_R\enveloping{L}$ can be identified 
with the left invariant polydifferential operators on $\groupoid{L}$. 	

The universal enveloping algebra $\enveloping{L}$ of the Lie algebroid $L\to M$ 
is a coalgebra over $R$ --- see~\cite{MR1815717}. 
Its comultiplication
\[ \Delta:\enveloping{L}\to\enveloping{L}\otimes_R\enveloping{L} \] 
is compatible with its filtration \eqref{Watermael} and characterized by the identities 
\begin{gather*}
\Delta(1)=1\otimes 1; \\ 
\Delta(b)=1\otimes b+b\otimes 1, \quad \forall b\in \sections{L}; \\ 
\Delta(u\cdot v)=\Delta(u)\cdot\Delta(v), \quad \forall u,v\in\enveloping{L} ,
\end{gather*}
where $1\in R$ denotes the constant function on $M$ with value $1$ 
while the symbol $\cdot$ denotes the multiplication in $\enveloping{L}$. 
We refer the reader to~\cite{MR1815717} for the precise meaning of the last equation above. 
Explicitly, we have
\begin{multline*} 
\Delta(b_1\cdot b_2\cdot\cdots\cdot b_n)= 1\otimes(b_1\cdot b_2\cdot\cdots\cdot b_n) 
+ \sum_{\substack{p+q=n \\ p,q\in\NN}}\sum_{\sigma\in\shuffle{p}{q}} 
(b_{\sigma(1)}\cdot\cdots\cdot b_{\sigma(p)}) \otimes (b_{\sigma(p+1)}\cdot\cdots\cdot b_{\sigma(n)}) \\
+ (b_1\cdot b_2\cdot\cdots\cdot b_n)\otimes 1
,\end{multline*} 
for all $b_1,\dots,b_n\in\sections{L}$. 

\subsection{Poincaré--Birkhoff--Witt isomorphisms}

Let $(L,A)$ be a Lie pair over $\KK$. 
Writing $\enveloping{L}\sections{A}$ for the left ideal of $\enveloping{L}$ generated by $\sections{A}$, 
the quotient $\frac{\enveloping{L}}{\enveloping{L}\sections{A}}$ is automatically a filtered $R$-coalgebra since 
\[ \Delta\big(\enveloping{L}\sections{A}\big)\subseteq \enveloping{L}\otimes_R \big(\enveloping{L}\sections{A}\big) 
+ \big(\enveloping{L}\sections{A}\big) \otimes_R \enveloping{L} \]
and the filtration \eqref{Watermael} on $\enveloping{L}$ descends to a filtration 
\[ \cdots \into \left(\frac{\enveloping{L}}{\enveloping{L}\sections{A}}\right)^{\leqslant n-1} 
\into \left(\frac{\enveloping{L}}{\enveloping{L}\sections{A}}\right)^{\leqslant n} 
\into \left(\frac{\enveloping{L}}{\enveloping{L}\sections{A}}\right)^{\leqslant n+1} \into \cdots \]
of $\frac{\enveloping{L}}{\enveloping{L}\sections{A}}$.

Likewise, deconcatenation defines an $R$-coalgebra structure on $\sections{S(L/A)}$.
The comultiplication 
\[ \Delta:\sections{S(L/A)}\to\sections{S(L/A)}\otimes_R\sections{S(L/A)} \] is given by 
\begin{multline*} 
\Delta(b_1\odot b_2\odot\cdots\odot b_n)=
1\otimes(b_1\odot b_2\odot\cdots\odot b_n) \\ 
\shoveright{+ \sum_{\substack{p+q=n \\ p,q\in\NN}}\sum_{\sigma\in\shuffle{p}{q}} 
(b_{\sigma(1)}\odot\cdots\odot b_{\sigma(p)}) \otimes (b_{\sigma(p+1)}\odot\cdots\odot b_{\sigma(n)})} \\
+ (b_1\odot b_2\odot\cdots\odot b_n)\otimes 1
,\end{multline*} 
for all $n\in\NN$ and $b_1,\dots,b_n\in\sections{L/A}$. 
The symbol $\odot$ denotes the symmetric product in $\sections{S(L/A)}$. 

The following theorem, which was obtained in~\cite{arXiv:1408.2903}, is an
extension of the classical Poincaré--Birkhoff--Witt isomorphism to Lie pairs. 

\begin{theorem}[{\cite[Theorem 2.1]{arXiv:1408.2903}}]
\label{Nairobi}
Let $(L,A)$ be a Lie pair. 
Given a splitting $j:L/A\to L$ of the short exact sequence $0\to A \to L \to L/A \to 0$, 
and a $L$-connection $\nabla$ on $L/A$ extending the Bott $A$-representation, 
there exists a unique isomorphism of filtered $R$-coalgebras 
\[ \pbw^{}:\sections{S(L/A)}\to\tfrac{\enveloping{L}}{\enveloping{L}\sections{A}} \] 
satisfying 
\begin{align}
\pbw^{}(f) &= f, \label{pbwzero} \\ 
\pbw^{}(b) &= j(b), \label{pbwun} \\
\pbw^{}(b^{n+1}) &= j(b)\cdot\pbw^{}(b^n)-\pbw^{}\big(\nabla_{j(b)}(b^n)\big) \label{pbwpower}
\end{align}
for all $f\in R$, $b\in \sections{L/A}$, and $n\in\NN$.
\end{theorem}

\begin{remark} 
Equation~\eqref{pbwpower} is equivalent to 
\begin{multline}\label{Rome}
\pbw^{}(b_0 \odot \cdots \odot b_n) =\tfrac{1}{n+1}\sum_{i=0}^{n} \Big( j(b_i) \cdot 
\pbw^{}(b_0 \odot \cdots \odot \widehat{b_i} \odot \cdots \odot b_n ) \\ -
\pbw^{}\big(\nabla_{j(b_i)} ( b_0 \odot \cdots \odot \widehat{b_i} \odot \cdots \odot b_n ) \big) \Big) 
\end{multline}
for all $b_0,\dots,b_n \in \sections{L/A}$.
\end{remark}

It is immediate that Equations~\eqref{pbwzero}, \eqref{pbwun}, and~\eqref{Rome} 
together define inductively a unique $R$-linear map $\pbw^{}$. 

The following lemma will be needed later on.

\begin{lemma}\label{Glaciere} 
For all $Y,Z$ in $\sections{L/A}$, we have 
\[ \pbw(Y\odot Z) = j(Y)\cdot j(Z) -j\big(\nabla_{j(Y)} Z\big) 
+\tfrac{1}{2}\ j\circ\beta^\nabla(Y,Z) .\] 
\end{lemma}

\begin{remark}
When $L=T_M$ and $A$ is the trivial Lie subalgebroid of $L$ of rank 0, 
the $\pbw^{}$ map of Theorem~\ref{Nairobi} is the inverse of the so-called `complete symbol map,' 
which is an isomorphism from the space $\enveloping{T_M}$ of differential operators on $M$ 
to the space $\sections{S(T_M)}$ of fiberwise polynomial functions on $T^\vee_M$. 
The complete symbol map was generalized to arbitrary Lie algebroids over $\RR$ 
by Nistor--Weinstein--Xu~\cite{MR1687747}. 
It played an important role in quantization theory \cite{MR706215,MR1687747, MR1231231, MR1427063}.
\end{remark}

\section{Fedosov dg manifolds for Lie pairs}

Given a Lie pair $(L,A)$ with quotient $B=L/A$, 
the graded manifold $L[1]\oplus B$ can be endowed 
with a homological vector field. 
We give two equivalent constructions of this homological vector field.

\subsection{First construction by way of the PBW map}

Making use of the Poincaré--Birkhoff--Witt isomorphism $\pbw$ 
of Theorem~\ref{Nairobi},
one can endow the graded manifold $L[1]\oplus B$ 
with a homological vector field.

Recall that every choice of a splitting $j:B\to L$ 
of the short exact sequence of vector bundles \eqref{Sydney}
and an $L$-connection $\nabla$ on $B$ extending the Bott $A$-connection
determines a Poincaré--Birkhoff--Witt map 
\[ \pbw^{}:\sections{SB}\to\tfrac{\enveloping{L}}{\enveloping{L}\sections{A}} ,\] 
which is an isomorphism of filtered $R$-coalgebras according to Theorem~\ref{Nairobi}. 

Being a quotient of the universal enveloping algebra $\enveloping{L}$ by a left ideal, 
the $R$-coalgebra $\tfrac{\enveloping{L}}{\enveloping{L}\sections{A}}$ is naturally a left $\enveloping{L}$-module. 
Hence $\tfrac{\enveloping{L}}{\enveloping{L}\sections{A}}$ is endowed with 
a canonical infinitesimal $L$-action by coderivations. 
Pulling back this infinitesimal action through $\pbw$, 
we obtain an infinitesimal $L$-action on $\sections{S(B)}$ by coderivations. 
The latter defines a \emph{flat} $L$-connection $\nabla^\lightning$ on $S(B)$: 
\begin{equation}\label{Vermont}
\nabla^\lightning_l(s)=\pbw^{-1}\big(l\cdot\pbw(s)\big) 
,\end{equation} 
for all $l\in\sections{L}$ and $s\in\sections{S B}$.

The $L$-connection $\nabla^\lightning$ on $S(B)$ induces
an $L$-connection on the dual bundle $\hat{S}(B^\vee)$ 
--- see Remark~\ref{Amalfi}.
We denote the corresponding Chevalley--Eilenberg differential by
\begin{equation}
d_L^{\nabla^\lightning}: \sections{\Lambda^\bullet L^\vee\otimes\hat{S}(B^\vee)} 
\to\sections{\Lambda^{\bullet+1} L^\vee\otimes\hat{S}(B^\vee)}
.\end{equation}
Since the covariant derivative 
\[ \nabla^\lightning_l:\sections{SB}\to\sections{SB} \] 
is a \emph{co}derivation of $\sections{SB}$ for all $l\in\sections{L}$,
the covariant derivative 
\[ \nabla^\lightning_l:\sections{\hat{S}(B^\vee)}\to\sections{\hat{S}(B^\vee)} \]
is a derivation of the symmetric algebra $\sections{\hat{S}(B^\vee)}$.
Note however that $\nabla^\lightning_l$ need not be a derivation 
of $\sections{SB}$ for any $l\in\sections{L}$.
Therefore, the operator $d_L^{\nabla^\lightning}$ on 
$\sections{\Lambda^\bullet L^\vee\otimes\hat{S}(B^\vee)}$
is a derivation of degree $(+1)$ satisfying 
$d_L^{\nabla^\lightning}\circ d_L^{\nabla^\lightning}=0$, 
i.e.\ it is a homological vector field on $L[1]\oplus B$.

\begin{proposition}\label{pro:PBW}
Given a Lie pair $(L,A)$ with quotient $B=L/A$, the choice of 
(1) a splitting $j:L/A\to L$ of the short exact sequence $0\to A \to L \to B \to 0$
and (2) an $L$-connection $\nabla$ on $B$ extending the Bott representation 
determines an operator $d_L^{\nabla^\lightning}$ as above 
making $(L[1]\oplus B,d_L^{\nabla^\lightning})$ a dg manifold.
\end{proposition}

\begin{remark}
The Kapranov dg manifolds of~\cite[Theorem 5.7]{arXiv:1408.2903} 
inspired the construction of the dg manifold of Proposition~\ref{pro:PBW}.
Indeed, the Kapranov dg manifold $(A[1]\oplus B,D)$ constructed 
in~\cite[Theorem 5.7]{arXiv:1408.2903} is a dg submanifold 
of the dg manifold $(L[1]\oplus B,d_L^{\nabla^\lightning})$ 
of Proposition~\ref{pro:PBW} as can be readily observed by comparing 
Equation~\eqref{Vermont} with~\cite[Equation~(46)]{arXiv:1408.2903}.
\end{remark}

We will see in Theorem~\ref{thm:Lodz} that, 
when $\nabla$ is torsion-free,
the homological vector field $d_L^{\nabla^\lightning}$ 
is exactly the homological vector field $Q$ 
constructed by Fedosov's iteration method 
as described in the next section.

\subsection{Second construction by way of Fedosov's iteration method}
\label{section:pek}

First we need to introduce some operators.

Let $r$ be the rank of the bundle $B$, 
let $q:L\to B$ be the canonical projection, 
and let $q\transpose:B^\vee\to L^\vee$ denote the dual map.
Given a local frame $\{\chi_k\}_{k=1}^r$ for the vector bundle $B^\vee$
and a multi-index $J=(J_1,J_2,\cdots,J_r)\in\NN_0^r$, 
we use the symbol $\chi^J$ to denote 
\[ \chi^J=\underset{J_1 \text{factors}}{\underbrace{\chi_1\odot\cdots\odot\chi_1}}
\odot\underset{J_2 \text{factors}}{\underbrace{\chi_2\odot\cdots\odot\chi_2}} 
\odot\cdots\odot\underset{J_r \text{factors}}{\underbrace{\chi_r\odot\cdots\odot\chi_r}} .\]
The multi-index $(0,\cdots,0,1,0,\cdots,0)$ having 
its single nonzero entry in $m$-th position is denoted $e_m$.

Consider the endomorphism $\delta$ of the vector bundle 
$\Lambda L^\vee\otimes\hat{S}(B^\vee)$ defined by 
\[ \delta(\omega\otimes\chi^J)=\sum_{m=1}^r \big(q\transpose(\chi_m)\wedge\omega\big)\otimes J_m\,\chi^{J-e_m} ,\] 
for all $\omega\in\Lambda L^\vee$ and $J\in\NN_0^r$
--- we declare that $J_m\chi^{J-e_m}=0$ if $J_m=0$.

The operator $\delta$ is a derivation of degree $(+1)$ 
of the bundle of graded commutative algebras 
$\Lambda^\bullet L^\vee\otimes\hat{S}(B^\vee)$ 
and satisfies $\delta^2=0$. 
The resulting cochain complex
\[ \begin{tikzcd}[column sep=small] 
\cdots \arrow{r}{} & \Lambda^{n-1} L^\vee \otimes\hat{S}(B^\vee) \arrow{r}{\delta} &
\Lambda^{n} L^\vee \otimes\hat{S}(B^\vee) \arrow{r}{\delta} & 
\Lambda^{n+1} L^\vee \otimes\hat{S}(B^\vee) \arrow{r}{} & \cdots 
\end{tikzcd} \] 
deformation retracts onto the trivial complex
\[ \begin{tikzcd} 
\cdots \arrow{r}{} & \Lambda^{n-1} A^\vee \arrow{r}{0} & 
\Lambda^{n} A^\vee \arrow{r}{0} & \Lambda^{n+1} A^\vee \arrow{r}{} & \cdots 
\end{tikzcd} \] 

Choose a splitting $i\circ p+j\circ q=\id_L$ of the short exact sequence 
\begin{equation} \label{Split} 
\begin{tikzcd} 
0 \arrow{r} & A \arrow{r}{i} & L \arrow[bend left, dashed]{l}{p} \arrow{r}{q} & 
B \arrow{r} \arrow[bend left, dashed]{l}{j} & 0 
\end{tikzcd} 
\end{equation} 
and its dual 
\[ \begin{tikzcd} 
0 \arrow{r} & B^\vee \arrow{r}{q\transpose} & 
L^\vee \arrow[bend left, dashed]{l}{j\transpose} \arrow{r}{i\transpose} 
& A^\vee \arrow{r} \arrow[bend left, dashed]{l}{p\transpose} & 0 
\end{tikzcd} .\]

Consider the chain maps
\begin{gather*} 
\tau:\Lambda^\bullet A^\vee \to\Lambda^\bullet L^\vee \otimes\hat{S}(B^\vee)
\\ \intertext{and} 
\sigma:\Lambda^\bullet L^\vee \otimes\hat{S}(B^\vee)\to\Lambda^\bullet A^\vee
\end{gather*}
respectively defined by 
\[ \tau(\alpha)=p\transpose(\alpha)\otimes 1 ,\] 
for all $\alpha\in\Lambda^\bullet A^\vee$, and 
\begin{equation}\label{Nuuk} \sigma(\omega\otimes\chi^J)=\begin{cases} \omega\otimes\chi^J & 
\text{if } v=0 \text{ and } \abs{J}=0 \\ 
0 & \text{otherwise,} \end{cases} \end{equation}
for all $\omega\in p\transpose(\Lambda^u A^\vee)\wedge q\transpose(\Lambda^v B^\vee)\subset\Lambda^{u+v}L^\vee$ 
and all multi-indices $J\in\NO^r$.

Finally, denoting any pair of dual local frames for $B$ and $B^\vee$ 
by $(\partial_i)_{i\in\{1,\dots,r\}}$ and $(\chi_j)_{j\in\{1,\dots,r\}}$ respectively,
consider the operator 
\[ h:\Lambda^{\bullet}L^\vee\otimes\hat{S}(B^\vee) \to\Lambda^{\bullet-1}L^\vee\otimes\hat{S}(B^\vee) \] 
defined by 
\begin{equation}\label{Skopje} 
h(\omega\otimes\chi^J)=\begin{cases} 
\frac{1}{v+\abs{J}}\sum_{k=1}^r (\interior{j(\partial_k)}\omega)\otimes\chi^{J+e_k} & \text{if } v\geqslant 1 \\ 
0 & \text{if } v=0 \end{cases} 
\end{equation}
for all $\omega\in p\transpose(\Lambda^u A^\vee)\wedge q\transpose(\Lambda^v B^\vee)\subset\Lambda^{u+v}L^\vee$ 
and all multi-indices $J\in\NO^r$.

\begin{remark}
Unlike $\delta$, the operator $h$ is \emph{not} a derivation 
of the bundle of graded commutative algebras 
$\Lambda^\bullet L^\vee\otimes\hat{S}(B^\vee)$. 
\end{remark}

The maps $\delta$, $\sigma$, $\tau$, and $h$ respect 
the exhaustive, complete, descending filtrations 
\begin{gather*} 
\mathcal{F}_0\supset\mathcal{F}_1\supset\mathcal{F}_2\supset\mathcal{F}_3\supset\cdots 
\\ \intertext{and} 
\mathscr{F}_0\supset\mathscr{F}_1\supset\mathscr{F}_2\supset\mathscr{F}_3\supset\cdots 
\end{gather*}
on $\Lambda^\bullet A^\vee$ and $\Lambda^\bullet L^\vee\otimes\hat{S}(B^\vee)$
defined by 
\begin{gather*} 
\mathcal{F}_m=\bigoplus_{k\geqslant m}\Lambda^k A^\vee 
\\ \intertext{and} 
\mathscr{F}_m=\prod_{k+p\geqslant m}\big(\Lambda^k L^\vee\otimes S^p(B^\vee)\big) 
.\end{gather*} 

Most significantly, the maps $\delta$, $\sigma$, $\tau$, and $h$ 
satisfy the following five identities: 
\begin{equation} \label{lem:contraction}
\sigma\tau=\id, \qquad \id-\tau\sigma= h\delta+\delta h, 
\qquad h\tau=0, \qquad \sigma h=0, \qquad h^2 =0 
.\end{equation}

Hence, we have proved the following 
\begin{proposition}
\label{Geneva}
The maps $\delta$, $h$, $\sigma$, and $\tau$ defined above determine a filtered contraction
\[ \begin{tikzcd}
(\Lambda^\bullet A^\vee,0) \arrow[r, "\tau", shift left] & 
(\Lambda^\bullet L^\vee\otimes\hat{S}(B^\vee),-\delta) \arrow[l, "\sigma", shift left] 
\arrow[loop, "h", out=5, in=-5, looseness = 3]
\end{tikzcd} .\]
\end{proposition}

A straighforward computation leads to the following

\begin{lemma}\label{Brescia}
Let $(L,A)$ be a Lie pair, 
let $\nabla$ be an $L$-connection on $B$ extending the Bott $A$-representation,
and let $\beta^\nabla$ be its torsion --- see Section~\ref{eq:dog}. 
Then $\beta^\nabla=0$ if and only if 
$\delta d_L^{\nabla}+d_L^{\nabla}\delta=0$.
\end{lemma}

Consider the four maps $\etendu{\delta}$, $\etendu{\sigma}$, $\etendu{h}$, and $\etendu{\tau}$
\[ \begin{tikzcd}
\sections{\Lambda^\bullet A^\vee \otimes B}
\arrow[shift right]{r}[swap]{\etendu{\tau}}
& \sections{\Lambda^{\bullet} L^\vee \otimes\hat{S}(B^\vee)\otimes B} 
\arrow[shift right]{l}[swap]{\etendu{\sigma}}
\arrow[shift left]{r}{\etendu{\delta}} 
& \sections{\Lambda^{\bullet+1} L^\vee \otimes\hat{S}(B^\vee)\otimes B}
\arrow[shift left]{l}{\etendu{h}} 
\end{tikzcd} \]
defined by 
\begin{align*} 
\etendu{\delta}(\omega\otimes\sigma\otimes b)&=\delta(\omega\otimes\sigma)\otimes b
& \etendu{\sigma}(\omega\otimes\sigma\otimes b)&=\sigma(\omega\otimes\sigma)\otimes b
\\ \etendu{h}(\omega\otimes\sigma\otimes b)&=h(\omega\otimes\sigma)\otimes b 
& \etendu{\tau}(\alpha\otimes b)&=\tau(\alpha)\otimes b
,\end{align*} 
for all $\alpha\in\sections{\Lambda A^\vee}$, 
$\omega\in\sections{\Lambda L^\vee}$, $\sigma\in\sections{\hat{S}(B^\vee)}$, 
and $b\in\sections{B}$.

It follows immediately from the identities \eqref{lem:contraction} that 
\[ \etendu{\sigma}\etendu{\tau}=\id, \qquad \id-\etendu{\tau}\etendu{\sigma}
= \etendu{h}\etendu{\delta}+\etendu{\delta}\etendu{h}, 
\qquad \etendu{h}\etendu{\tau}=0, \qquad \etendu{\sigma}\etendu{h}=0, 
\qquad \etendu{h}^2 =0 .\]

We are now ready to present a second construction
of a holomorphic vector field $Q$ on the graded manifold $L[1]\oplus B$, 
which relies on Fedosov's iteration method.

\begin{proposition}\label{strawberry}
Let $(L,A)$ be a Lie pair.
Given a splitting $i\circ p+j\circ q=\id_L$ of the short exact sequence \eqref{Split}
and a torsion-free $L$-connection $\nabla$ on $B$,
there exists a \emph{unique} 1-form valued in 
the formal vertical vector fields on $B$:
\[ X^\nabla\in\sections{L^\vee\otimes
\hat{S}^{\geqslant 2}(B^\vee)\otimes B} \] 
satisfying $\etendu{h}(X^\nabla)=0$ 
and such that the derivation
\[ Q: \sections{\Lambda^\bullet L^\vee\otimes\hat{S}(B^\vee)} 
\to \sections{\Lambda^{\bullet+1} L^\vee\otimes\hat{S}(B^\vee)} \] 
defined by 
\[ Q=-\delta+d_L^\nabla+X^\nabla \] 
satisfies $Q^2=0$. 
Here $X^\nabla$ acts on the algebra 
$\sections{\Lambda^\bullet L^\vee\otimes\hat{S}(B^\vee)}$ 
as a derivation in a natural fashion.
As a consequence, $(L[1]\oplus B, Q)$ is a dg manifold.
\end{proposition}

\begin{proof}
Suppose there exists such a $\xnabla$ 
and consider its decomposition $\xnabla=\sum_{k=2}^{\infty} X_k$, 
where $X_k\in\sections{L^\vee\otimes S^{k}(B^\vee)\otimes B}$. 
Then $Q=-\delta+d_L^\nabla+\xtwo+\xhigher$ 
with $\xhigher=\sum_{k=3}^{\infty} X_k$ and 
\[ \begin{split} 
Q^2 =&\ \delta^2 
- \big(\delta d_L^\nabla+d_L^\nabla\delta\big) 
+ \big\{d_L^\nabla d_L^\nabla-\delta\xtwo-\xtwo\delta\big\} \\ 
&\ + \big\{d_L^\nabla\xnabla+\xnabla d_L^\nabla+{\xnabla}^2 
-\delta\xhigher-\xhigher\delta\big\} \\ 
=&\ \delta^2 - \lie{\delta}{d_L^\nabla} 
+ \big\{ R^\nabla-\lie{\delta}{\xtwo} \big\}
+ \big\{ \lie{d_L^\nabla+\tfrac{1}{2}\xnabla}{\xnabla}
- \lie{\delta}{\xhigher} \big\} 
.\end{split} \] 

Let us write $\Lambda^p\otimes S^q$ for 
$\sections{\Lambda^p L^\vee \otimes S^q B^\vee}$.

Since 
\[ \begin{tikzcd}[column sep=huge,row sep=small] 
& \Lambda^{p+1}\otimes S^{q-1} \\ 
\Lambda^{p}\otimes S^{q} 
\arrow{ru}[near end]{-\delta} 
\arrow{r}[near end]{d_L^\nabla} 
\arrow{rd}[swap, near end]{\xtwo} 
\arrow{rdd}[swap, near end]{\xhigher} 
& \Lambda^{p+1}\otimes S^{q} \\ 
& \Lambda^{p+1}\otimes S^{q+1} \\ 
& \Lambda^{p+1}\otimes \hat{S}^{\geqslant q+2} 
,\end{tikzcd} \]
we have 
\[ \begin{tikzcd}[column sep=huge,row sep=small] 
&& \Lambda^{p+2}\otimes S^{q-2} \\ 
&& \Lambda^{p+2}\otimes S^{q-1} \\ 
\Lambda^{p}\otimes S^{q} 
\arrow{rruu}[near end]{\delta^2} 
\arrow{rru}[swap, near end]{\lie{\delta}{d_L^\nabla}} 
\arrow{rr}[swap, near end]{R^\nabla-\lie{\delta}{\xtwo}} 
\arrow{rrd}[swap, near end]{\lie{d_L^\nabla+\frac{1}{2}\xnabla}{\xnabla}
-\lie{\delta}{\xhigher}}
&& \Lambda^{p+2}\otimes S^{q} \\ 
&& \Lambda^{p+2}\otimes \hat{S}^{\geqslant q+1} 
.\end{tikzcd} \]

Since $\delta^2=0$ and $\lie{\delta}{d_L^\nabla}=0$ (by Lemma~\ref{Brescia}), 
we obtain the commutative diagram 
\[ \begin{tikzcd}[row sep=small]
& \Lambda^{p}\otimes S^{q} 
\arrow{dl}[swap]{R^\nabla-\lie{\delta}{\xtwo}} 
\arrow{dd}{Q^2}
\arrow{dr}{\lie{d_L^\nabla+\frac{1}{2}\xnabla}{\xnabla}
-\lie{\delta}{\xhigher}} & \\ 
\Lambda^{p+2}\otimes S^{q} & & \Lambda^{p+2}\otimes\hat{S}^{\geqslant q+1} \\
& \Lambda^{p+2}\otimes\hat{S}^{\geqslant q} \arrow[two heads]{lu} \arrow[two heads]{ru} & 
\end{tikzcd} \]
The requirement $Q^2=0$ is thus equivalent to the pair of equations 
\begin{align*}
\lie{\delta}{\xtwo} &= R^\nabla \\ 
\lie{\delta}{\xhigher} &= \lie{d_L^\nabla+\tfrac{1}{2}\xnabla}{\xnabla}
\end{align*}

Note that $\etendu{\sigma}(\xtwo)=0$ and $\etendu{\sigma}(\xhigher)=0$ 
since $\xtwo,\xhigher\in\sections{L^\vee\otimes\hat{S}^{\geqslant 2}(B^\vee)\otimes B}$ 
and also that $\etendu{h}(\xtwo)=0$ and $\etendu{h}(\xhigher)=0$ as $\etendu{h}(\xnabla)=0$. 
Therefore, since $\etendu{\delta}\etendu{h}+\etendu{h}\etendu{\delta}=\id-\etendu{\tau}\etendu{\sigma}$, 
we obtain $\etendu{h}\etendu{\delta}(\xtwo)=\xtwo$ and $\etendu{h}\etendu{\delta}(\xhigher)=\xhigher$. 

It follows that 
\begin{align*}
\xtwo&=\etendu{h}\etendu{\delta}(\xtwo)=\etendu{h}(\lie{\delta}{\xtwo})=\etendu{h}(R^\nabla) \\ 
\xhigher&=\etendu{h}\etendu{\delta}(\xhigher)=\etendu{h}(\lie{\delta}{\xhigher})=\etendu{h} 
\lie{d_L^\nabla+\tfrac{1}{2}\xnabla}{\xnabla}
.\end{align*}
Projecting the second equation onto 
$\sections{L^\vee\otimes S^{k+1}(B^\vee)\otimes B}$, 
we obtain 
\begin{align}
X_2 &= \etendu{h}(R^\nabla) \label{deux} \\ 
X_{k+1} &= \etendu{h}\Big(d_L^\nabla\circ X_k+X_k\circ d_L^\nabla
+\sum_{\substack{p+q=k+1 \\ 2\leqslant p,q\leqslant k-1}} X_p\circ X_q\Big), 
\qquad\text{for}\ k\geqslant 2. \label{troisetplus}
\end{align} 
The successive terms of $\xnabla=\sum_{k=2}^{\infty} X_k$ 
can thus be computed sequentially starting from $X_2=\etendu{h}(R^\nabla)$. 
Therefore, if it exists, the derivation $X^\nabla$ is uniquely determined 
by the torsion-free connection $\nabla$ and the splitting $j:B\to L$.

Now, defining $X_k$ inductively by the relations \eqref{deux} and \eqref{troisetplus} 
and setting $\xnabla=\sum_{k=2}^{\infty} X_k$, 
we have $\etendu{h}(\xnabla)=\etendu{h}(\xtwo+\xhigher)=\etendu{h}^2\big(R^\nabla+\etendu{\delta}(\xhigher)\big)=0$
since $\etendu{h}^2=0$. 
Moreover, we have $X_2=\etendu{h}(R^\nabla)\in\sections{L^\vee\otimes S^2(B^\vee)\otimes B}$ 
as $R^\nabla\in\sections{\Lambda^2 L^\vee\otimes B^\vee\otimes B}$. 
Making use of Equation~\eqref{troisetplus}, one proves by induction on $k$ that 
$X_k\in\sections{L^\vee\otimes S^k(B^\vee)\otimes B}$. 
This completes the proof of the existence of $\xnabla$. 
\end{proof}

We elect to call a dg manifold $(L[1]\oplus B,Q)$ 
constructed in this way a \emph{Fedosov dg manifold}.

We note that Proposition~\ref{strawberry}
was proved independently by Batakidis--Voglaire~\cite{MR3724780}.

\subsection{Equivalence of the two constructions}

The aim of this section is to prove the following theorem, 
which is one of the main result of this paper.

\begin{theorem}\label{thm:Lodz}
Let $(L,A)$ be a Lie pair, 
let $i\circ p+j\circ q=\id_L$ be a splitting of the short exact sequence \eqref{Split},
and let $\nabla$ be an $L$-connection on $B$ extending the Bott $A$-connection.
If $\nabla$ is torsion-free, then the dg manifold 
$(L[1]\oplus B,d_L^{\nabla^\lightning})$ described in Proposition~\ref{pro:PBW} 
coincides with the dg manifold $(L[1]\oplus B,Q)$ constructed 
by the Fedosov iteration described in (the proof of) Proposition~\ref{strawberry}.
\end{theorem}

Consider the bundle map \[ \Theta: L\otimes SB\to SB \] 
defined in~\cite{arXiv:1408.2903} by the relation
\[ \Theta (l;s)=\nabla^\lightning_l s -\nabla_l s-q(l)\odot s, 
\quad\forall l\in\sections{L},s\in\sections{SB} .\]

\begin{lemma}[{\cite[Lemma~5.13]{arXiv:1408.2903}}]
\label{Oregon}
For all $l\in\sections{L}$, we have $\Theta (l;1)=0$.
\end{lemma}

\begin{proposition}[{\cite[Lemma 5.16]{arXiv:1408.2903}}]
\label{Utah}
For all $l\in\sections{L}$, 
the map $s\mapsto\Theta (l;s)$ is a coderivation of 
the $R$-coalgebra $\sections{SB}$ which preserves the filtration 
\[ \cdots \into\Gamma( S^{\leqslant n-1}B) \into
\Gamma( S^{\leqslant n}B) \into \Gamma(S^{\leqslant n+1} B) \into \cdots .\]
\end{proposition}

\begin{lemma}[{\cite[Lemma 5.17]{arXiv:1408.2903}}]
\label{Mozambique}
For all $n\in\NN$ and all $b_0,b_1,\dots,b_n\in\sections{B}$, we have 
\[ \sum_{k=0}^n \Theta\big(j(b_k);b_0\odot\cdots\odot
\widehat{b_k}\odot\cdots\odot b_n\big) =0 .\]
\end{lemma}

\begin{lemma}\label{Idaho} 
Provided the $L$-connection $\nabla$ on $B$ extends the Bott representation, 
the bundle map $\Theta$ associated with the connection $\nabla$ 
and a splitting $i\circ p+j\circ q=\id_L$ satisfies the relation
\[ \Theta(l;b)=-\tfrac{1}{2}\beta^\nabla(q(l),b) \]
for all $l\in\sections{L}$ and $b\in\sections{B}$.
\end{lemma}

In particular, if the $L$-connection $\nabla$ on $B$ is torsion-free, 
we have $\Theta(l;b)=0$ for all $l\in\sections{L}$ and $b\in\sections{B}$.

\begin{proof}
Since $\pbw(b)=j(b)$ for all $b\in\sections{B}$, 
we may rewrite Lemma~\ref{Glaciere} as 
\[ j(Y)\cdot\pbw(Z) = \pbw\big(Y\odot Z + \nabla_{j(Y)} Z 
-\tfrac{1}{2}\ \beta^\nabla(Y,Z)\big) \] 
or 
\[ \nabla^{\lightning}_{j(Y)} Z= Y\odot Z + \nabla_{j(Y)} Z 
-\tfrac{1}{2}\ \beta^\nabla(Y,Z) \]
for all $Y,Z\in\sections{B}$.

On the other hand, according to~\cite[Lemma~5.11]{arXiv:1408.2903}, 
we have \[ \nabla^{\lightning}_{i(X)} Z= \nabla^{\Bott}_X Z \] 
for all $X\in\sections{A}$ and $Z\in\sections{B}$.
Furthermore, since the $L$-connection $\nabla$ on $B$ 
is assumed to extend the Bott representation, 
we have \[ \nabla^{\Bott}_X Z = \nabla_{i(X)} Z \] 
for all $X\in\sections{A}$ and $Z\in\sections{B}$.

Therefore, for all $l\in\sections{L}$ and $b\in\sections{B}$, we have
\begin{multline*} \nabla^\lightning_l b 
= \nabla^\lightning_{i\circ p(l)} b + \nabla^\lightning_{j\circ q(l)} b \\
= \nabla_{i\circ p(l)} b + \big\{ q(l)\odot b + \nabla_{j\circ q(l)} b 
-\tfrac{1}{2}\ \beta^\nabla(q(l),b) \big\} \\
= q(l)\odot b + \nabla_{l} b 
-\tfrac{1}{2}\ \beta^\nabla(q(l),b)
.\end{multline*}
The result then follows from the definition of $\Theta$.
\end{proof}

Let \[ \sections{S B}\otimes_{R}\sections{\hat{S}(B^\vee)}\xto{\duality{-}{-}} R \] 
be the duality pairing defined by 
\[ \duality{b_1\odot\cdots\odot b_p}
{\beta_1\odot\cdots\odot\beta_q} = 
\begin{cases} 
\sum_{\sigma\in S_p} 
\interior{b_1}\beta_{\sigma(1)}\cdot
\interior{b_2}\beta_{\sigma(2)}\cdots 
\interior{b_p}\beta_{\sigma(p)} 
& \text{if } p=q \\ 
0 & \text{if } p\neq q 
\end{cases} \] 
for all $b_1,\dots,b_p\in\sections{B}$ 
and $\beta_1,\dots,\beta_q\in\sections{B^\vee}$. 

A straightforward computation yields the following

\begin{lemma}\label{Bercy}
Let $(\partial_i)_{i\in\{1,\dots,r\}}$ be a local frame of $B$ 
and let $(\chi_j)_{j\in\{1,\dots,r\}}$ be the dual local frame of $B^\vee$. 
We have 
\[ \duality{\partial^I}{\chi^J}=I!\ \delta_{I,J} 
,\quad\forall I,J\in\NO^n \] 
and 
\[ \sigma=\sum_{I\in\NO^n}\frac{1}{I!}
\duality{\partial^I}{\sigma}\, \chi^I 
,\quad\forall \sigma\in\Gamma\big(\hat{S}(B^\vee)\big) .\] 
\end{lemma}

\begin{lemma}\label{Goncourt}
For any $l\in \sections{L}$,
and all $s\in\sections{S B }$, 
and $\sigma\in\Gamma\big(\hat{S} B^\vee \big)$, we have
\[ \duality{s}{\interior{l}\delta(\sigma)}
=\duality{q(l)\odot s}{\sigma} .\] 
\end{lemma}

\begin{proof}
It suffices to prove the relation for $s=\partial^I$ and $\sigma=\chi^J$. 
We have 
\begin{multline*}
\duality{\partial^I}{\interior{l}\delta(\chi^J)} 
= \duality{\partial^I}{\sum_{k=1}^r \interior{q(l)}\chi_k \cdot J_k \chi^{J-e_k}} 
= \sum_{k=1}^r \interior{q(l)}\chi_k \ \duality{\partial^I}{J_k \chi^{J-e_k}} \\ 
= \sum_{k=1}^r \interior{q(l)}\chi_k \ J_k\ I!\ \delta_{I,J-e_k} 
= \sum_{k=1}^r \interior{q(l)}\chi_k \ J!\ \delta_{I+e_k,J} 
= \sum_{k=1}^r \interior{q(l)}\chi_k \ \duality{\partial^{I+e_k}}{\chi^J} \\ 
= \duality{\sum_{k=1}^r \interior{q(l)}\chi_k \cdot \partial_k\odot\partial^I}{\chi^J} 
= \duality{q(l)\odot\partial^I}{\chi^J}
.\qedhere\end{multline*}
\end{proof}

Consider the map 
\[ \Xi : \sections{S^k B^\vee } \to 
\sections{L^\vee\otimes\hat{S}^{\geqslant k+1}(B^\vee)}, 
\quad \forall k\geqslant 0, \]
defined by 
\begin{equation}\label{Rambuteau} 
\duality{s}{\interior{l}\Xi (\sigma)} 
= \duality{\Theta (l; s)}{\sigma} 
,\end{equation}
for all $l\in\sections{L}$, $s\in\sections{S B}$, 
and $\sigma\in\sections{\hat{S}(B^\vee)}$. 

The $L$-connections $\nabla$ and $\nabla^\lightning$ 
defined on $S(B)$ induce $L$-connections 
on the dual bundle $\hat{S}(B^\vee)$ --- see Remark~\ref{Amalfi}.

\begin{proposition}\label{Budapest}
$d_L^{\nabla^\lightning}=-\delta+d_L^{\nabla}-\Xi$
\end{proposition}

\begin{proof} 
According to Remark~\ref{Amalfi}, we have 
\[ \duality{\nabla^\lightning_l s}{\sigma}+
\duality{s}{\nabla^\lightning_l \sigma} = 
\rho(l) \duality{s}{\sigma} =
\duality{\nabla_l s}{\sigma}+
\duality{s}{\nabla_l \sigma} ,\]
for all $l\in\sections{L}$, $s\in\sections{SB}$, and $\sigma\in\sections{\hat{S}(B^\vee)}$.
From there, we obtain 
\begin{align*}
\duality{\nabla^\lightning_l s-\nabla_l s}{\sigma} &=
\duality{s}{\nabla_l \sigma-\nabla^\lightning_l \sigma} \\ 
\duality{q(l)\odot s+\Theta(l; s)}{\sigma} &= 
\duality{s}{\interior{l}\big(d_L^\nabla\sigma-d_L^{\nabla^\lightning}\sigma\big)} \\ 
\intertext{and, making use of Lemma~\ref{Goncourt} 
and Equation~\eqref{Rambuteau},} 
\duality{s}{\interior{l}\delta(\sigma)+\interior{l}\Xi(\sigma)\big)} &= 
\duality{s}{\interior{l}\big(d_L^\nabla\sigma-d_L^{\nabla^\lightning}\sigma\big)} 
\end{align*}
or, equivalently, \[ d_L^{\nabla^\lightning}=-\delta+d_L^\nabla-\Xi .\qedhere\] 
\end{proof}

\begin{proposition}
For every $l\in\sections{L}$, 
the operator $\interior{l}\Xi$ is a derivation of the algebra 
the $R$-algebra $\Gamma\big(\hat{S}(B^\vee)\big)$ which preserves the filtration 
\[ \cdots \into \Gamma\big(\hat{S}^{\geqslant n+1}(B^\vee)\big) 
\into \Gamma\big(\hat{S}^{\geqslant n}(B^\vee)\big) 
\into \Gamma\big(\hat{S}^{\geqslant n-1}(B^\vee)\big) \into \cdots .\]
\end{proposition}

\begin{proof}
The result follows immediately from Proposition~\ref{Utah} 
since the algebra $\Gamma\big(\hat{S}(B^\vee)\big)$ is dual to the coalgebra 
$\Gamma(S(B))$ and $\interior{l}\Xi$ 
is the transpose of 
$\Theta(l;\argument)$ according to Equation~\eqref{Rambuteau}. 
\end{proof}

Therefore, $\Xi$ may be regarded as an element of the subspace
$\Gamma\big(L^\vee\otimes\hat{S}(B^\vee)\otimes B\big)$ 
of the space of derivations of the algebra 
$\Gamma\big(\Lambda^\bullet L^\vee\otimes\hat{S}(B^\vee)\big)$. 

Given any pair of dual local frames 
$(\partial_i)_{i\in\{1,\dots,r\}}$ 
and $(\chi_j)_{j\in\{1,\dots,r\}}$ 
for the vector bundles $B$ and $B^\vee$ 
and any pair of dual local frames 
$(l_m)_{m=1}^{\rk(L)}$ 
and $(\lambda_m)_{m=1}^{\rk(L)}$ 
for the vector bundles $L$ and $L^\vee$, 
we have 
\[ \Xi=\sum_{m=1}^{\rk(L)}\sum_{k=1}^r 
\lambda_m\otimes\interior{l_m}\Xi(\chi_k)\otimes\partial_k ,\]
since each $\interior{l_m}\Xi$ 
is a derivation of the $R$-algebra
$\Gamma\big(\hat{S}(B^\vee)\big)$, 
which is generated locally by $\chi_1,\dots,\chi_r$.
Furthermore, for all $l\in\sections{L}$, we have 
\begin{align*} 
\interior{l}\Xi (\chi_k) =& \sum_{I\in\NO^r} \frac{1}{I!} 
\duality{\partial^I}{\interior{l}\Xi (\chi_k)} \chi^I
&&\text{by Lemma~\ref{Bercy},} \\ 
=& \sum_{I\in\NO^r} \frac{1}{I!} 
\duality{\Theta(l;\partial^I)}{\chi_k} \chi^I
&&\text{by Equation~\eqref{Rambuteau}.} 
\end{align*}

\begin{lemma}\label{Oslo}
\strut
\begin{enumerate}
\item If the $L$-connection $\nabla$ on $B$ is torsion-free, then 
$\Xi\in\Gamma\big(L^\vee\otimes\hat{S}^{\geqslant 2}(B^\vee)\otimes B\big)$.
\item Otherwise,
$\Xi\in\Gamma\big(L^\vee\otimes\hat{S}^{\geqslant 1}(B^\vee)\otimes B\big)$.
\end{enumerate}
\end{lemma}

\begin{proof} 
If the connection $\nabla$ is torsion-free, 
then $\interior{l}\Xi\in\Gamma\big(\hat{S}^{\geqslant 2}(B^\vee)\otimes B\big)$
as $\Theta(l;\partial^I)=0$ for $\abs{I}\leqslant 1$ 
according to Lemma~\ref{Oregon} and Lemma~\ref{Idaho}. 
However, if the torsion of $\nabla$ is not zero, we can only say that 
$\interior{l}\Xi\in\Gamma\big(\hat{S}^{\geqslant 1}(B^\vee)\otimes B\big)$
as $\Theta(l;\partial^I)$ need not vanish for $\abs{I}=1$.
\end{proof}

We note that, for every pair of dual local frames $(\partial_i)_{i\in\{1,\dots,r\}}$ and 
$(\chi_j)_{j\in\{1,\dots,r\}}$ for $B$ and $B^\vee$, 
we have 
\[ \Xi=\sum_{k=1}^r \sum_{J\in\NO^r} 
\frac{1}{J!} \duality{\partial^J}{\Xi (\chi_k)} 
\chi^J \partial_k .\] 

\begin{lemma}\label{Ragusa}
For all $\lambda\in\sections{L^\vee}$ and $J\in\NO^r$, we have \[ h(\lambda\otimes\chi^J) 
= \frac{1}{1+\abs{J}}\sum_{k=1}^r 
\interior{j(\partial_k)}\lambda\otimes\chi^{J+e_k} ,\] 
where $(\partial_i)_{i\in\{1,\dots,r\}}$ is any local frame of $B$ 
and $(\chi_j)_{j\in\{1,\dots,r\}}$ is the dual local frame of $B^\vee$.
\end{lemma}

\begin{proof}
For $\lambda\in\sections{q\transpose B^\vee}$, 
the result follows immediately from Equation~\eqref{Skopje}, 
the very definition of $h$. 
The results holds for $\lambda\in\sections{p\transpose A^\vee}$ as well since, for all $\alpha\in\sections{A^\vee}$, 
we have $h(p\transpose(\alpha)\otimes\chi^J) =0$ by the very definition of $h$ 
and $\interior{j(\partial_k)}p\transpose(\alpha)=0$ as $p\circ j=0$. 
\end{proof}

\begin{proposition}\label{Telegraphe}
$\etendu{h}(\Xi)=0$
\end{proposition}

\begin{proof}
Let $(\partial_i)_{i\in\{1,\dots,r\}}$ be a local frame of $B$ 
and let $(\chi_j)_{j\in\{1,\dots,r\}}$ be the dual local frame of $B^\vee$.

From \[ \Xi=\sum_{k=1}^r \sum_{J\in\NO^r} \frac{1}{J!} 
\duality{\partial^J}{\Xi(\chi_k)} \chi^J \partial_k ,\] 
we obtain, using Lemma~\ref{Ragusa}, 
\begin{align*}
\etendu{h}(\Xi) =& 
\sum_{k=1}^r \sum_{J\in\NO^r} 
h\bigg\{ \frac{1}{J!}\duality{\partial^J}{\Xi(\chi_k)} \chi^J\bigg\} \partial_k \\
=& \sum_{k=1}^r \sum_{J\in\NO^r} 
\frac{1}{1+\abs{J}} \sum_{p=1}^r \frac{1}{J!} 
\duality{\partial^J}{\interior{j(\partial_p)}\Xi(\chi_k)} \chi_p \chi^J \partial_k \\ 
=& \sum_{k=1}^r \sum_{J\in\NO^r} 
\frac{1}{1+\abs{J}} \sum_{p=1}^r \frac{1}{J!} 
\duality{\Theta(j(\partial_p); \partial^J)}{\chi_k} 
\chi^{J+e_p} \partial_k \\ 
=& \sum_{k=1}^r \sum_{M\in\NO^r} 
\frac{1}{\abs{M}} \frac{1}{M!} 
\duality{ \sum_{p=1}^r M_p\ \Theta (j(\partial_p); \partial^{M-e_p})}{\chi_k} 
\chi^M \partial_k 
.\end{align*}
It follows directly from Lemma~\ref{Mozambique} that 
\[ \sum_{p=1}^r M_p \ \Theta (j(\partial_p); \partial^{M-e_p}) =0 \]
for every $M=(M_1,\dots,M_r)\in\NO^r$. 
\end{proof}

We are now ready to complete the proof of Theorem~\ref{thm:Lodz}.

\begin{proof}[Proof of Theorem~\ref{thm:Lodz}] 
Since $\Xi\in\Gamma\big(L^\vee\otimes
\hat{S}^{\geqslant 2}(B^\vee)\otimes B\big)$ 
provided $\beta^\nabla=0$ (Proposition~\ref{Oslo}), 
$\etendu{h}(\Xi)=0$ (Proposition~\ref{Telegraphe}), 
and $d_L^{\nabla^\lightning}=-\delta+d_L^{\nabla}-\Xi$ 
(Proposition~\ref{Budapest}) satisfies 
$d_L^{\nabla^\lightning}\circ d_L^{\nabla^\lightning}=0$, 
the uniqueness statement in Proposition~\ref{strawberry} 
asserts that $X^\nabla=-\Xi$ and $Q=d_L^{\nabla^\lightning}$. 
\end{proof}

\section{Dolgushev--Fedosov quasi-isomorphisms}

\subsection{Contraction of the Fedosov dg manifold}

Our second main result, Theorem~\ref{thm:main2} below, 
extends the Dolgushev--Fedosov quasi-isomorphism \cite[Theorem 3]{MR2102846} 
to the context of Lie pairs. This section is devoted to its proof;
Theorem~\ref{thm:main2} is an immediate consequence 
of Theorem~\ref{thm:Lodz} and Proposition~\ref{Ottignies} below.
 
\begin{theorem} 
\label{thm:main2}
Given a Lie pair $(L,A)$, let $d^{\nabla^\lightning}$ 
be the homological vector field on $L[1]\oplus B$ 
determined by the choice of 
a splitting $i\circ p+j\circ q=\id_L$ 
of the short exact sequence \eqref{Split}
and an $L$-connection $\nabla$ on $B$
as in Proposition~\ref{pro:PBW}.
Then the natural inclusion $(A[1],d_A)\into(L[1]\oplus B,Q)$
is a quasi-isomorphism of dg manifolds.
\end{theorem}

\begin{remark}
In particular, if $L$ is the tangent bundle to a smooth manifold $M$ 
and $A$ is its rank-zero subbundle, 
Theorem~\ref{thm:main2} reduces to the part of \cite[Theorem~3]{MR2102846} 
pertaining to functions on the manifold $M$.
\end{remark}

Dolgushev established the quasi-isomorphism 
\cite[Theorem 3]{MR2102846} by a direct verification.
Here we will prove the stronger result that the cochain complexes
$\big(\sections{\Lambda^\bullet L^\vee\otimes\hat{S}(B^\vee)}, 
d_L^{\nabla^\lightning}\big)$
and $\big(\sections{\Lambda^\bullet A^\vee },d_A\big)$ 
are indeed homotopy equivalent. 
Homological perturbation --- see Appendix~\ref{Vilnius} --- 
provides a quick and easy proof of this result: 
the operator $\varrho=d_L^{\nabla^\lightning}+\delta$
can be understood as a perturbation of the cochain complex 
$\big(\sections{\Lambda^\bullet L^\vee\otimes\hat{S}(B^\vee)},-\delta\big)$
appearing in the contraction of Proposition~\ref{Geneva}.

\begin{lemma}\label{Zurich}
The operator $\varrho=d_L^{\nabla^\lightning}+\delta$ 
is a perturbation of the cochain complex 
\[ \begin{tikzcd}[column sep=scriptsize] 
\cdots \arrow{r}{} & \sections{\Lambda^{k-1} L^\vee 
\otimes\hat{S}(B^\vee)} \arrow{r}{-\delta} &
\sections{\Lambda^{k} L^\vee\otimes\hat{S}(B^\vee)} \arrow{r}{-\delta} & 
\sections{\Lambda^{k+1} L^\vee\otimes\hat{S}(B^\vee)} \arrow{r}{} & \cdots 
\end{tikzcd} \] 
raising the descending filtration 
$\cdots\subseteq\mathscr{F}_{m+1}\subseteq\mathscr{F}_m\subseteq\mathscr{F}_{m-1}\subseteq\cdots$
defined by
\[ \mathscr{F}_m=\prod_{k+p\geqslant m} \Gamma\big(\Lambda^k L^\vee\otimes S^p B^\vee\big) .\]
\end{lemma}

\begin{proof}
According to Proposition~\ref{Budapest}, 
we have $\varrho=d_L^{\nabla^\lightning}+\delta=d_L^{\nabla}-\Xi$.
Furthermore, according to Lemma~\ref{Oslo} and 
the very definitions of $\delta$, $d^\nabla_L$, and $\Xi$, we have
\[ \begin{tikzcd}[column sep=huge,row sep=small]
& \sections{\Lambda^{k+1}(L^\vee)\otimes S^{p-1}(B^\vee)} \\
\sections{\Lambda^k(L^\vee)\otimes S^p(B^\vee)}
\arrow[bend left]{ru}{-\delta}
\arrow{r}{d_L^{\nabla}}
\arrow[bend right]{rd}[swap]{-\Xi}
& \sections{\Lambda^{k+1}(L^\vee)\otimes S^p(B^\vee)} \\
& \sections{\Lambda^{k+1}(L^\vee)\otimes\hat{S}^{\geqslant p}(B^\vee)}
\end{tikzcd} .\]
Therefore, while the differential $-\delta$ respects the filtration $\mathscr{F}$, 
the operator $\varrho=d_L^{\nabla}-\Xi$ raises it: 
$\varrho(\mathscr{F}_m)\subseteq\mathscr{F}_{m+1}$.
Finally, we have $(-\delta+\varrho)^2=(d_L^{\nabla^\lightning})^2=0$
since the connection $\nabla^\lightning$ is flat.
This concludes the proof.
\end{proof}

\begin{proposition}\label{Ottignies}
Given a Lie pair $(L,A)$, let $d_A$ denote the Chevalley--Eilenberg differential 
of the Lie algebroid $A$ regarded as a homological vector field on $A[1]$ and 
let $d^{\nabla^\lightning}$ be the homological vector field on $L[1]\oplus B$ 
determined by the choice of a splitting $i\circ p+j\circ q=\id_L$ 
of the short exact sequence \eqref{Split} and an $L$-connection $\nabla$ on $B$
as in Proposition~\ref{strawberry}.
Then, there exists a contraction
\[ \begin{tikzcd}
\big(\sections{\Lambda^\bullet A^\vee},d_{A}\big)
\arrow[r, "\perturbed{\tau}", shift left] &
\big(\sections{\Lambda^\bullet L^\vee\otimes\hat{S}(B^\vee)},d_{L}^{\nabla^\lightning}\big) 
\arrow[l, "\sigma", shift left]
\arrow[loop, "\perturbed{h}", out=5, in=-5, looseness = 3]
\end{tikzcd} .\]
where 
\[ \perturbed{\tau}=\sum_{k=0}^\infty(h\varrho)^k\tau, 
\qquad \perturbed{h}=\sum_{k=0}^\infty(h\varrho)^k h, 
\qquad \varrho=d_L^{\nabla^\lightning}+\delta ,\] 
and the maps $\delta$, $\tau$, $\sigma$, and $h$ are 
those defined in Section~\ref{section:pek}. 
In particular, $\sigma$ is the natural inclusion of graded manifolds 
$A[1]\into L[1]\oplus B$ defined by Equation~\eqref{Nuuk}.
\end{proposition}

\begin{proof}
We proceed by homological perturbation (see Lemma~\ref{Riga}). 
Starting from the filtered contraction of Proposition~\ref{Geneva}, 
it suffices to perturb the coboundary operator $-\delta$ by the operator $\varrho$ (see Lemma~\ref{Zurich})
to obtain the new contraction 
\[ \begin{tikzcd}
\big(\sections{\Lambda^\bullet A^\vee},\vartheta\big)
\arrow[r, "\perturbed{\tau}", shift left] &
\big(\sections{\Lambda^\bullet L^\vee\otimes\hat{S}(B^\vee)},-\delta+\varrho\big) 
\arrow[l, "\perturbed{\sigma}", shift left]
\arrow[loop, "\perturbed{h}", out=5, in=-5, looseness = 3]
\end{tikzcd} .\]
We have $\sigma\varrho h=0$ as, for all $n,p\in\NO$, 
\[ \begin{tikzcd}[column sep=small] 
\sections{\Lambda^n L^\vee \otimes S^p B^\vee} 
\arrow{r}{h} & 
\sections{\Lambda^{n-1} L^\vee \otimes S^{p+1} B^\vee} \arrow{r}{\varrho} & 
\sections{\Lambda^n L^\vee \otimes\hat{S}^{\geqslant p+1}(B^\vee)} \arrow{r}{\sigma} & 0.
\end{tikzcd} \] 
Therefore, we obtain 
\[ \perturbed{\sigma}:=\sum_{k=0}^\infty\sigma(\varrho h)^k=\sigma \]
and 
\[ \vartheta:=\sum_{k=0}^\infty\sigma(\varrho h)^k\varrho\tau=\sigma\varrho\tau
=\sigma(d_L^\nabla-\Xi^\nabla)(p\transpose\otimes 1) =\sigma\big((d_L\circ p\transpose)\otimes 1\big)=d_A .\]
The result follows immediately since 
$-\delta+\varrho=d_L^{\nabla^\lightning}$ (Proposition~\ref{Budapest}). 
\end{proof}

We note that a similar construction of a Fedosov resolution of the algebra 
of smooth functions on a manifold based on homological perturbation 
was described by Hans-Christian Herbig in~\cite{arXiv:0708.3598}.
 
\subsection{Matched pairs}

In this section, we establish an explicit expression for the quasi-isomorphism 
\[ \perturbed{\tau}:\sections{\Lambda^{\bullet}A^\vee} \to \sections{\Lambda^{\bullet}L^\vee\otimes\hat{S}(B^\vee)} \] 
defined in Proposition~\ref{Ottignies} valid only in the special case of matched pairs. 
This formula can be considered as an extension to the matched pair case of the augmentation map 
\[ \perturbed{\tau}:C^\infty(M)\to\Omega^0\big(M;\hat{S}(T_M^{\vee})\big) \]
arising from the Emmrich--Weinstein `formal exponential map' --- see Section~\ref{EW} and~\cite[Theorem~1.6]{MR1327535}.

Suppose the short exact sequence \eqref{Split} admits a splitting $j:B\to L$ 
whose image $j(B)$ is a Lie subalgebroid of $L$ --- i.e.\ $L=A\bowtie B$ is a matched pair. 
Then $B$ is a Lie algebroid and composition of the morphism of associative algebras 
$\enveloping{B}\to\enveloping{L}$ induced by $j$ with the canonical projection 
$\enveloping{L}\onto\frac{\enveloping{L}}{\enveloping{L}\sections{A}}$ 
yields a canonical isomorphism of left $R$-coalgebras $\enveloping{B}\cong\frac{\enveloping{L}}{\enveloping{L}\sections{A}}$. 

Since $L=A\bowtie B$ is a matched pair, we have a Bott $B$-representation 
on $A$: \[ \nabla^{\Bott}_b a = p\big(\lie{j(b)}{i(a)}\big) ,\quad\forall 
b\in\sections{B},a\in\sections{A} .\] 
The dual $B$-connection on $A^\vee$ extends to the exterior algebra $\Lambda A^\vee$ by derivation: 
\[ \nabla^{\Bott}_b \alpha = i\transpose\big( \liederivative{j(b)} (p\transpose\alpha) \big) 
,\quad\forall b\in\sections{B},\alpha\in\sections{\Lambda A^\vee} .\] 
One can show that, for every $b\in\sections{B}$, the diagram 
\[ \begin{tikzcd}
\sections{\Lambda^\bullet A^\vee} \arrow{r}{\nabla^{\Bott}_{b}} \arrow{d}{p\transpose} & \sections{\Lambda^\bullet A^\vee} \arrow{d}{p\transpose} \\ 
\sections{\Lambda^\bullet L^\vee} \arrow{r}[swap]{\liederivative{j(b)}} & \sections{\Lambda^\bullet L^\vee}
\end{tikzcd} \]
commutes. 
Since the Bott $B$-connection on $\Lambda A^\vee$ is flat, $\sections{\Lambda A^\vee}$ is a left $\enveloping{B}$-module: the action 
\[ \enveloping{B}\times\sections{\Lambda A^\vee}\xto{\rtimes}\sections{\Lambda A^\vee} \] satisfies 
\[ p\transpose\big( b_1 b_2\cdots b_n \rtimes\alpha \big) = \liederivative{j(b_1)} \liederivative{j(b_2)} 
\cdots \liederivative{j(b_n)}(p\transpose \alpha) ,\] 
for all $b_1,b_2,\dots,b_n\in\sections{B}$ and $\alpha\in\sections{\Lambda A^\vee}$. 

In the matched pair case, the chain map 
$\perturbed{\tau}:\sections{\Lambda^{\bullet}A^\vee} \to \sections{\Lambda^{\bullet}L^\vee\otimes\hat{S}(B^\vee)}$ 
defined in Proposition~\ref{Ottignies} 
admits a simple description in terms of the splitting $j:B\into L$, 
the associated left $\enveloping{B}$-module structure $\rtimes$ on $\sections{\Lambda A^\vee}$, 
and the map $\pbw:\sections{S(B)}\to\enveloping{B}$ induced by $j$ and $\nabla$.

Consider the derivation $\mathscr{D}$ of the subalgebra
$\sections{p\transpose(\Lambda A^\vee)\otimes\hat{S}(B^\vee)}$ 
of $\sections{\Lambda L^\vee\otimes\hat{S}(B^\vee)}$ defined by 
\[ \mathscr{D}(p\transpose\alpha\otimes\chi^J)=
\sum_{k=1}^{r} \left\{ p\transpose(\nabla^{\Bott}_{\partial_k}\alpha) 
\otimes\chi_k\cdot\chi^J +p\transpose\alpha\otimes\chi_k\cdot
\nabla_{j(\partial_k)}(\chi^J) \right\} ,\]
for all $\alpha\in\sections{\Lambda^{\bullet}A^\vee}$ and $J\in\NO^{r}$. 

\begin{remark}
A analogue of the derivation $\mathscr{D}$ was introduced recently in~\cite[Section~2.1 and Remark~2.3]{arXiv:1512.07863}. 
It would be interesting to understand the precise relation between these two derivations. 
\end{remark}

The remainder of this section is devoted to the proof of the following theorem. 
\begin{theorem}\label{Montmartre}
If the splitting 
\[ \begin{tikzcd} 
0 \arrow{r} & A \arrow{r}{i} & L \arrow[bend left, dashed]{l}{p} \arrow{r}{q} 
& B \arrow{r} \arrow[bend left, dashed]{l}{j} & 0 
\end{tikzcd} \]
identifies $B$ with a Lie subalgebroid of $L$, 
then $\sections{\Lambda A^\vee}$ is a left $\enveloping{B}$-module and 
the chain map \[ \perturbed{\tau}:\sections{\Lambda^{\bullet}A^\vee} 
\to \sections{\Lambda^{\bullet}L^\vee\otimes\hat{S}(B^\vee)} \] defined in Proposition~\ref{Ottignies} satisfies 
\[ \perturbed{\tau}=\exp(\mathscr{D})\circ\tau \] 
and 
\begin{equation}
\label{eq:FCO}
 \perturbed{\tau}(\alpha)
= \sum_{J\in\NO^r} \frac{1}{J!} p\transpose\big( \pbw(\partial^J)\rtimes 
\alpha\big)\otimes\chi^J ,\qquad\forall \alpha\in\sections{\Lambda A^\vee} .
\end{equation}
\end{theorem}

The following lemma is an analogue of Lemma~\ref{Telegraphe}, 
which is proved \textit{mutans mutandis}.
 
\begin{lemma}\label{Telegraphe2}
For all $\alpha\in\sections{\Lambda A^\vee}$ and $J\in\NO^{r}$, 
we have $h\Xi(p\transpose\alpha\otimes\chi^J)=0$. 
\end{lemma}

\begin{lemma}\label{Montrouge}
For all $\alpha\in\sections{\Lambda A^\vee}$ and $J\in\NO^r$, we have 
\begin{align*} 
h \varrho(p\transpose\alpha\otimes\chi^J) 
=& \frac{1}{1+\abs{J}}\sum_k\left\{ p\transpose(\nabla^{\Bott}_{\partial_k}\alpha) 
\otimes\chi_k\cdot\chi^J +p\transpose\alpha\otimes\chi_k\cdot \nabla_{j(\partial_k)}(\chi^J) \right\} \\ 
=& \frac{1}{1+\abs{J}} \mathscr{D}(p\transpose\alpha\otimes\chi^J) 
.\end{align*}
\end{lemma}

\begin{proof}
Since $L=A\bowtie B$, Proposition~\ref{Piacenza} asserts that, if $\alpha\in\sections{\Lambda^u A^\vee}$, then
\[ d_L(p\transpose\alpha)\in\Omega^{u+1,0}\oplus\Omega^{u,1} ,\] 
where $\Omega^{u,v}=\sections{p\transpose(\Lambda^{u} A^\vee)\wedge q\transpose(\Lambda^{v} B^\vee)}$. 
Therefore, if $\alpha\in\sections{\Lambda^u A^\vee}$, we have 
\[ d_L^{\nabla}(p\transpose\alpha\otimes\chi^J) = d_L(p\transpose\alpha)\otimes\chi^J 
+\sum_t\lambda_t\wedge(p\transpose\alpha)\otimes\nabla_{l_t}(\chi^J) 
\in (\Omega^{u+1,0}\oplus\Omega^{u,1})\otimes_R\sections{S^{\abs{J}}(B^\vee)} \]
and it follows from Equation~\eqref{Skopje} that 
\begin{align*} hd_L^{\nabla}(p\transpose\alpha\otimes\chi^J) 
=& \frac{1}{1+\abs{J}} \sum_k \bigg\{ i_{j(\partial_k)}d_L(p\transpose\alpha)\otimes\chi^{J+e_k} 
+\sum_t(i_{j(\partial_k)}\lambda_t)\cdot p\transpose\alpha\otimes\chi_k\cdot\nabla_{l_t}(\chi^J) \bigg\} \\
=& \frac{1}{1+\abs{J}} \sum_k \bigg\{ \liederivative{j(\partial_k)}(p\transpose\alpha)\otimes\chi^{J+e_k} 
+p\transpose\alpha\otimes\chi_k\cdot\nabla_{\sum_t(i_{j(\partial_k)}\lambda_t) l_t}(\chi^J) \bigg\} \\
=& \frac{1}{1+\abs{J}} \sum_k \bigg\{ p\transpose(\nabla^{\Bott}_{\partial_k}\alpha)\otimes\chi^{J+e_k} 
+p\transpose\alpha\otimes\chi_k\cdot\nabla_{j(\partial_k)}(\chi^J) \bigg\} 
.\end{align*}
The desired result follows from Lemma~\ref{Telegraphe2}. 
\end{proof}

\begin{proof}[Proof of Theorem~\ref{Montmartre}]
Reasoning by induction on $k$, one proves that 
\[ \mathscr{D}^k\circ\tau(\alpha)\in\sections{p\transpose(\Lambda A^\vee)\otimes 
S^k(B^\vee)} \] for all $\alpha\in\sections{\Lambda A^\vee}$ and $k\in\NN$. 
Using Lemma~\ref{Montrouge} and reasoning by induction on $k$ once again, 
one proves that \[ (h\varrho)^{k}\circ\tau=\frac{1}{k!}\mathscr{D}^k\circ\tau \] 
for all $k\in\NN$. It follows that 
\[ \breve{\tau} = \bigg(\sum_{k=0}^{\infty}(h\varrho)^k\bigg)\circ\tau 
= \bigg(\sum_{k=0}^{\infty}\frac{1}{k!}\mathscr{D}^k\bigg)\circ\tau 
= \exp(\mathscr{D})\circ\tau .\] 

Set $\machin{\alpha}=\sum_{J\in\NO^r} \frac{1}{J!}p\transpose\big( \pbw(\partial^J)\rtimes \alpha\big)\otimes\chi^J$ 
for all $\alpha\in\sections{\Lambda A^\vee}$. 
We claim that $(\id-h\varrho)\machin{\alpha}=\tau(\alpha)$. 
It then follows from Proposition~\ref{Fort-de-France} that 
$\breve{\tau}(\alpha)=\machin{\alpha}$ --- the desired result. 

It remains to establish our claim. 
From 
\[ 0=\rho(j(\partial_k))\duality{\partial^K}{\chi^J}
=\duality{\nabla_{j(\partial_k)}(\partial^K)}{\chi^J}
+\duality{\partial^K}{\nabla_{j(\partial_k)}(\chi^J)} ,\] 
we obtain 
\begin{equation}\label{Caracas} \nabla_{j(\partial_k)}(\chi^J) 
= \sum_K\frac{1}{K!}\duality{\partial^K}{\nabla_{j(\partial_k)}(\chi^J)}\chi^K 
= -\sum_K\frac{1}{K!}\duality{\nabla_{j(\partial_k)}(\partial^K)}{\chi^J}\chi^K 
.\end{equation}

From Lemma~\ref{Telegraphe2}, Lemma~\ref{Montrouge}, and Equation~\eqref{Caracas}, we obtain 
\begin{multline*} 
h\varrho\machin{\alpha} = 
\sum_{J}\tfrac{1}{J!}\tfrac{1}{1+\abs{J}}\sum_k\bigg\{ 
p\transpose\Big(\nabla^{\Bott}_{\partial_k} \big(\pbw(\partial^J)\rtimes\alpha\big)\Big)\otimes 
\chi_k\cdot\chi^J \\ 
-p\transpose\big(\pbw(\partial^J)\rtimes\alpha\big)\otimes\chi_k\cdot 
\sum_K\tfrac{1}{K!}\duality{\nabla_{j(\partial_k)}(\partial^K)}{\chi^J}\chi^K\bigg\} 
\end{multline*} 
This can be rewritten as 
\begin{multline*} 
h\varrho\machin{\alpha} = 
\sum_J\tfrac{1}{J!}\tfrac{1}{1+\abs{J}}\sum_k 
p\transpose\big(j(\partial_k)\cdot\pbw(\partial^J)\rtimes\alpha\big) 
\otimes\chi^{J+e_k} \\ 
-\sum_K\tfrac{1}{K!}\tfrac{1}{1+\abs{K}}\sum_k 
p\transpose\Big(\pbw\big(\nabla_{j(\partial_k)}(\partial^K)\big)\rtimes\alpha\Big) 
\otimes\chi^{K+e_k} 
\end{multline*}
and then 
\begin{multline*} 
h\varrho\machin{\alpha} = \sum_{\substack{M\in\NO^r \\ \abs{M}\geqslant 1}} 
\frac{1}{M!} p\transpose \bigg( \frac{1}{\abs{M}}\sum_k M_k 
\Big( j(\partial_k)\cdot\pbw\big(\partial^{M-e_k}\big) 
-\pbw\big(\nabla_{j(\partial_k)}\partial^{M-e_k}\big) \Big) 
\rtimes\alpha \bigg)\otimes\chi^M 
.\end{multline*}
Finally, it follows from Equation~\eqref{Rome} that 
\[ h\varrho\machin{\alpha} 
= \sum_{\substack{M\in\NO^r \\ \abs{M}\geqslant 1}} 
\frac{1}{M!} p\transpose \big(\pbw(\partial^M)\rtimes\alpha\big)\otimes\chi^M 
= \machin{\alpha}- p\transpose(\alpha)\otimes 1 
= \machin{\alpha}-\tau(\alpha) .\qedhere\] 
\end{proof}

\section{Application: the `formal exponential map' of Emmrich--Weinstein}
\label{EW}

In this section, we give a simple and direct proof of a result 
--- see \cite[Theorem~1.6 and Section~7]{MR1327535} and Theorem~\ref{Nassau} below --- 
which Emmrich--Weinstein proved by way of an indirect argument 
involving real analytic manifolds.

According to Proposition~\ref{strawberry}, given a torsion-free affine connection $\nabla$ on $M$, 
one can construct a homological vector field $Q$ on the graded manifold $T_M[1]\oplus T_M$ by Fedosov's iteration method. 
This is essentially what \cite[Theorem~1.1]{MR1327535} --- more precisely the special case when $a=0$ --- 
and \cite[Theorem~2]{MR2102846} assert.
One obtains a derivation $Q$ of $\Omega^\bullet\big(M;\hat{S}(T^\vee_M)\big)$ such that $Q^2=0$.
Identifying $\sections{\hat{S}(T^\vee_M)}$ to the algebra of functions on the formal neighborhood 
$(T_M)_\infty$ of the zero section of the tangent bundle to $M$, 
Emmrich--Weinstein regard the derivation $Q$ as vector fields determining a distribution on 
$(T_M)_\infty$ transverse to the fibers of $(T_M)_\infty\to M$, 
i.e.\ a (nonlinear) formal Ehresmann connection on $(T_M)_\infty$.
Since $Q^2=0$, this distribution is involutive and the Ehresmann connection is flat.
Those sections $\varsigma\in\sections{\hat{S}(T^\vee_M)}$ such that $Q(\varsigma)=0$ 
are interpreted as functions on $(T_M)_\infty$ which are constant along the leaves 
of the foliation tangent to the flat Ehresmann connection.
As explained by Ehresmann--Weinstein, the Ehresmann connection is transverse to the zero section 
and each one of its leaves intersects the zero section in a unique point.
The leaves of the foliation given by the flat Ehresmann connection are the fibers 
of a `mapping' $\EXP:(T_M)_\infty\to M$, which Ehresmann--Weinstein call 
`formal exponential map.'
Identifying $M$ with the zero section of $(T_M)_\infty$, 
functions defined on $M$ can be extended to functions on $(T_M)_\infty$ constant along the leaves.
The resulting map $C^\infty(M)\to\sections{\hat{S}(T^\vee_M)}$ 
is the pull-back of functions through $\EXP$.

\begin{theorem}[{\cite[Theorem~1.6]{MR1327535}}]
\label{Nassau}
Given a torsion-free affine connection $\nabla$ on $M$, 
the `formal exponential map' $\EXP$ described above coincides 
with the infinite-order jet of the geodesic exponential map $\exp$
determined by the connection $\nabla$.
\end{theorem}

\begin{proof}
By definition, the `formal exponential map' $\EXP$ is completely determined by $Q$ in the following sense:
the pull-back $\varsigma=\EXP^*(f)$ of a function $f\in C^\infty(M)$ by $\EXP$ 
is the unique solution $\varsigma\in\sections{\hat{S}(T^\vee_M)}$ 
of the initial value problem \[ Q(\varsigma)=0 \qquad \sigma(\varsigma)=f .\]
We think of the map $\sigma:\sections{\hat{S}(T^\vee_M)}\to C^\infty(M)$ 
as the pull-back of functions through the zero section of $(T_M)_\infty\to M$.
On the other hand, Proposition~\ref{Ottignies} applied to the Lie pair $(L,A)$ 
where $L$ is the tangent bundle to $M$ and $A$ is its rank-zero subbundle yields the contraction 
\[ \begin{tikzcd}
C^\infty(M) \arrow[r, "\perturbed{\tau}", shift left] &
\big(\Omega^\bullet(M;\hat{S}(T^\vee_M)),d^{\nabla^\lightning}\big) 
\arrow[l, "\perturbed{\sigma}", shift left]
\arrow[loop, "\perturbed{h}", out=5, in=-5, looseness = 3]
\end{tikzcd} \]
where $C^\infty(M)$ is seen as a cochain complex concentrated in degree $0$.
In particular, for all $f\in C^\infty(M)$, we have 
\[ \perturbed{\tau}(f)\in\sections{\hat{S}(T^\vee_M)}, 
\qquad d^{\nabla^\lightning}\big(\perturbed{\tau}(f)\big)=0, 
\qquad\text{and}\qquad \sigma\big(\perturbed{\tau}(f)\big)=f .\]
According to Theorem~\ref{thm:Lodz}, we have $Q=d^{\nabla^\lightning}$. 
Therefore, $\EXP^*=\perturbed{\tau}$.
Let $n$ denote the dimension of the manifold $M$.
It follows from Equation~\eqref{eq:FCO} in Theorem~\ref{Montmartre} that
\[ \perturbed{\tau}(f) = \sum_{J\in\NO^n} \frac{1}{J!} \pbw(\partial^J)f \otimes \chi^J .\] 
According to~\cite[Theorem~3.11]{arXiv:1408.2903}, the Poincar\'e--Birkhoff--Witt isomorphism $\pbw$ 
(described in Theorem~\ref{Nairobi}) is the infinite-order jet of the geodesic exponential map $\exp:T_M\to M$ 
arising from the connection $\nabla$.
Hence, we obtain 
\[ \EXP^*(f)= \perturbed{\tau}(f) = \sum_{J\in\NO^n} \frac{1}{J!} \partial^J\big(\exp^*(f)\big) \otimes \chi^J .\]
This concludes the proof that $\EXP$ is the infinite-order jet of $\exp$.
\end{proof}

In fact, Theorem~\ref{Montmartre} above can be seen as an extension 
of Emmrich--Weinstein's result \cite[Theorem~1.6]{MR1327535}
to the case of matched pairs.

\appendix

\section{Homological perturbation}\label{Vilnius} 

A contraction of a cochain complex $(N,\delta)$ onto a cochain complex $(M,d)$
consists of a pair of chain maps $\sigma:N\to M$ and $\tau:M\to N$ 
and an endomorphism $h:N\to N[-1]$ of the graded module $N$ satisfying 
the following five relations: 
\begin{gather*} \sigma\tau=\id_M, \qquad \tau\sigma-\id_N=h\delta+\delta h, \\
\sigma h=0, \qquad h\tau=0, \qquad h^2=0 .\end{gather*}
We symbolize such a contraction by a diagram 
\[ \begin{tikzcd} 
(M,d) \arrow[r, "\tau", shift left] & 
(N,\delta) \arrow[l, "\sigma", shift left] 
\arrow["h", loop right]
\end{tikzcd} .\]
If, furthermore, the cochain complexes $N$ and $M$ are filtered and the maps $\sigma$, $\tau$, 
and $h$ preserve the filtration, the contraction is said to be filtered \cite[\S12]{MR0056295}. 

A descending filtration
\[ \cdots \subseteq F_{p+1}N \subseteq F_{p}N \subseteq F_{p-1}N \subseteq \cdots \]
on a cochain complex $N$ is said to be exhaustive if $N=\bigcup_p F_p N$
and complete if $N=\varprojlim_p \frac{N}{F_p N}$.

A \emph{perturbation} of the filtered cochain complex 
\[ \begin{tikzcd} \cdots \arrow{r} & N^{n-1} \arrow{r}{\delta} 
& N^{n} \arrow{r}{\delta} & N^{n+1} \arrow{r} & \cdots \end{tikzcd} \]
is an operator $\varrho$ of degree $+1$ on $N$, 
which raises the filtration (i.e.\ $\varrho(F_{p} N)\subseteq F_{p+1} N$)
and satisfies $(\delta+\varrho)^2=0$
so that $\delta+\varrho$ is a new coboundary operator on $N$.

We refer the reader to~\cite[\S1]{MR1109665} for a brief history of the following lemma. 

\begin{lemma}[Homological Perturbation \cite{MR0220273}]
\label{Riga}
Let 
\[ \begin{tikzcd} 
(M,d) \arrow[r, "\tau", shift left] & 
(N,\delta) \arrow[l, "\sigma", shift left] 
\arrow["h", loop right]
\end{tikzcd} \]
be a filtered contraction. 
Given a perturbation $\varrho$ of the cochain complex $(N,\delta)$, 
if the filtrations on $M$ and $N$ are exhaustive and complete, then the series 
\begin{align*}
\perturbed{\tau}&:=\sum_{k=0}^{\infty}(h\varrho)^k\tau &
\perturbed{h}&:=\sum_{k=0}^{\infty}(h\varrho)^k h=\sum_{k=0}^{\infty}h(\varrho h)^k \\
\perturbed{\sigma}&:=\sum_{k=0}^{\infty}\sigma(\varrho h)^k &
\vartheta&:=\sum_{k=0}^{\infty}\sigma(\varrho h)^k\varrho\tau=\sum_{k=0}^{\infty}\sigma\varrho(h\varrho)^k\tau
\end{align*}
converge, $\vartheta$ is a perturbation of the cochain complex $(M,d)$, and 
\[ \begin{tikzcd} 
(M,d+\vartheta) \arrow[r, "\perturbed{\tau}", shift left] & 
(N,\delta+\varrho) \arrow[l, "\perturbed{\sigma}", shift left] 
\arrow["\perturbed{h}", loop right]
\end{tikzcd} \]
constitutes a new filtered contraction. 
\end{lemma}

\begin{proposition}\label{Fort-de-France}
Under the same hypothesis as in Lemma~\ref{Riga}, 
the chain map $\perturbed{\tau}$ is entirely determined by $\tau$, $h\varrho$ and the relation $(\id-h\varrho)\perturbed{\tau}=\tau$. 
Likewise, the homotopy operator $\perturbed{h}$ is entirely determined by $h$, $h\varrho$ and the relation $(\id-h\varrho)\perturbed{h}=h$. 
\end{proposition}

\begin{proof}
Since the filtration on $N$ is complete and
$\varrho$ raises the filtration while $h$ preserves it, 
the geometric series $\sum_{k=0}^\infty (h\varrho)^k$ converges 
and its sum is the inverse of the operator $\id-h\varrho$. 
The result follows immediately.
\end{proof}

\section*{Acknowledgements} 

We would like to thank 
Ruggero Bandiera, 
Panagiotis Batakidis, 
Martin Bordemann, 
Damien Broka,
Vasily Dolgushev,
Olivier Elchinger,
Camille Laurent-Gengoux, 
Hsuan-Yi Liao, 
Kirill Mackenzie,
Rajan Mehta,
and Yannick Voglaire 
for fruitful discussions and useful comments. 
Stiénon is grateful to Université Paris~7 for its hospitality during his sabbatical leave in 2015--2016. 

\bibliographystyle{amsplain}
\bibliography{references}
\end{document}